\documentclass[12pt,reqno]{amsart}
\usepackage{amssymb,amscd,url}

\usepackage[all]{xy}

\begin{document}

\def\framebox#1{\text{\textbf{[\negthinspace[}#1\textbf{]\negthinspace]}}} 

\allowdisplaybreaks


\title[Shafarevich-type theorems for dynamical systems]
     {Good reduction and Shafarevich-type theorems for dynamical systems with portrait level structures}


\date{\today}
\author[Joseph H. Silverman]{Joseph H. Silverman}
\email{jhs@math.brown.edu}
\address{Mathematics Department, Box 1917
         Brown University, Providence, RI 02912 USA}
\thanks{Research supported by Simons Collaboration Grant \#241309}



\hyphenation{ca-non-i-cal semi-abel-ian}


\newtheorem{theorem}{Theorem}
\newtheorem{lemma}[theorem]{Lemma}
\newtheorem{sublemma}[theorem]{Sublemma}
\newtheorem{conjecture}[theorem]{Conjecture}
\newtheorem{proposition}[theorem]{Proposition}
\newtheorem{corollary}[theorem]{Corollary}
\newtheorem*{claim}{Claim}

\theoremstyle{definition}
\newtheorem*{definition}{Definition}
\newtheorem{example}[theorem]{Example}
\newtheorem{remark}[theorem]{Remark}
\newtheorem{question}[theorem]{Question}

\theoremstyle{remark}
\newtheorem*{acknowledgement}{Acknowledgements}


\newenvironment{notation}[0]{%
  \begin{list}%
    {}%
    {\setlength{\itemindent}{0pt}
     \setlength{\labelwidth}{4\parindent}
     \setlength{\labelsep}{\parindent}
     \setlength{\leftmargin}{5\parindent}
     \setlength{\itemsep}{0pt}
     }%
   }%
  {\end{list}}

\newenvironment{parts}[0]{%
  \begin{list}{}%
    {\setlength{\itemindent}{0pt}
     \setlength{\labelwidth}{1.5\parindent}
     \setlength{\labelsep}{.5\parindent}
     \setlength{\leftmargin}{2\parindent}
     \setlength{\itemsep}{0pt}
     }%
   }%
  {\end{list}}
\newcommand{\Part}[1]{\item[\upshape#1]}

\def\Case#1#2{%
\paragraph{\textbf{\boldmath Case #1: #2.}}\hfil\break\ignorespaces}

\renewcommand{\a}{\alpha}
\renewcommand{\b}{\beta}
\newcommand{\g}{\gamma}
\renewcommand{\d}{\delta}
\newcommand{\e}{\epsilon}
\newcommand{\f}{\varphi}
\newcommand{\bfphi}{{\boldsymbol{\f}}}
\renewcommand{\l}{\lambda}
\renewcommand{\k}{\kappa}
\newcommand{\lhat}{\hat\lambda}
\newcommand{\m}{\mu}
\newcommand{\bfmu}{{\boldsymbol{\mu}}}
\renewcommand{\o}{\omega}
\renewcommand{\r}{\rho}
\newcommand{\rbar}{{\bar\rho}}
\newcommand{\s}{\sigma}
\newcommand{\sbar}{{\bar\sigma}}
\renewcommand{\t}{\tau}
\newcommand{\z}{\zeta}

\newcommand{\D}{\Delta}
\newcommand{\G}{\Gamma}
\newcommand{\F}{\Phi}
\renewcommand{\L}{\Lambda}

\newcommand{\ga}{{\mathfrak{a}}}
\newcommand{\gb}{{\mathfrak{b}}}
\newcommand{\gn}{{\mathfrak{n}}}
\newcommand{\gp}{{\mathfrak{p}}}
\newcommand{\gP}{{\mathfrak{P}}}
\newcommand{\gq}{{\mathfrak{q}}}

\newcommand{\Abar}{{\bar A}}
\newcommand{\Ebar}{{\bar E}}
\newcommand{\kbar}{{\bar k}}
\newcommand{\Kbar}{{\bar K}}
\newcommand{\Pbar}{{\bar P}}
\newcommand{\Sbar}{{\bar S}}
\newcommand{\Tbar}{{\bar T}}
\newcommand{\gbar}{{\bar\gamma}}
\newcommand{\lbar}{{\bar\lambda}}
\newcommand{\ybar}{{\bar y}}
\newcommand{\phibar}{{\bar\f}}

\newcommand{\Acal}{{\mathcal A}}
\newcommand{\Bcal}{{\mathcal B}}
\newcommand{\Ccal}{{\mathcal C}}
\newcommand{\Dcal}{{\mathcal D}}
\newcommand{\Ecal}{{\mathcal E}}
\newcommand{\Fcal}{{\mathcal F}}
\newcommand{\Gcal}{{\mathcal G}}
\newcommand{\Hcal}{{\mathcal H}}
\newcommand{\Ical}{{\mathcal I}}
\newcommand{\Jcal}{{\mathcal J}}
\newcommand{\Kcal}{{\mathcal K}}
\newcommand{\Lcal}{{\mathcal L}}
\newcommand{\Mcal}{{\mathcal M}}
\newcommand{\Ncal}{{\mathcal N}}
\newcommand{\Ocal}{{\mathcal O}}
\newcommand{\Pcal}{{\mathcal P}}
\newcommand{\Qcal}{{\mathcal Q}}
\newcommand{\Rcal}{{\mathcal R}}
\newcommand{\Scal}{{\mathcal S}}
\newcommand{\Tcal}{{\mathcal T}}
\newcommand{\Ucal}{{\mathcal U}}
\newcommand{\Vcal}{{\mathcal V}}
\newcommand{\Wcal}{{\mathcal W}}
\newcommand{\Xcal}{{\mathcal X}}
\newcommand{\Ycal}{{\mathcal Y}}
\newcommand{\Zcal}{{\mathcal Z}}

\renewcommand{\AA}{\mathbb{A}}
\newcommand{\BB}{\mathbb{B}}
\newcommand{\CC}{\mathbb{C}}
\newcommand{\FF}{\mathbb{F}}
\newcommand{\GG}{\mathbb{G}}
\newcommand{\NN}{\mathbb{N}}
\newcommand{\PP}{\mathbb{P}}
\newcommand{\QQ}{\mathbb{Q}}
\newcommand{\RR}{\mathbb{R}}
\newcommand{\ZZ}{\mathbb{Z}}

\newcommand{\bfa}{{\boldsymbol a}}
\newcommand{\bfb}{{\boldsymbol b}}
\newcommand{\bfc}{{\boldsymbol c}}
\newcommand{\bfd}{{\boldsymbol d}}
\newcommand{\bfe}{{\boldsymbol e}}
\newcommand{\bff}{{\boldsymbol f}}
\newcommand{\bfg}{{\boldsymbol g}}
\newcommand{\bfi}{{\boldsymbol i}}
\newcommand{\bfj}{{\boldsymbol j}}
\newcommand{\bfp}{{\boldsymbol p}}
\newcommand{\bfr}{{\boldsymbol r}}
\newcommand{\bfs}{{\boldsymbol s}}
\newcommand{\bft}{{\boldsymbol t}}
\newcommand{\bfu}{{\boldsymbol u}}
\newcommand{\bfv}{{\boldsymbol v}}
\newcommand{\bfw}{{\boldsymbol w}}
\newcommand{\bfx}{{\boldsymbol x}}
\newcommand{\bfy}{{\boldsymbol y}}
\newcommand{\bfz}{{\boldsymbol z}}
\newcommand{\bfA}{{\boldsymbol A}}
\newcommand{\bfF}{{\boldsymbol F}}
\newcommand{\bfB}{{\boldsymbol B}}
\newcommand{\bfD}{{\boldsymbol D}}
\newcommand{\bfG}{{\boldsymbol G}}
\newcommand{\bfI}{{\boldsymbol I}}
\newcommand{\bfM}{{\boldsymbol M}}
\newcommand{\bfP}{{\boldsymbol P}}
\newcommand{\bfzero}{{\boldsymbol{0}}}
\newcommand{\bfone}{{\boldsymbol{1}}}

\newcommand{\Aut}{\operatorname{Aut}}
\newcommand{\Berk}{{\textup{Berk}}}
\newcommand{\Birat}{\operatorname{Birat}}
\newcommand{\codim}{\operatorname{codim}}
\newcommand{\Crit}{\operatorname{Crit}}
\newcommand{\critwt}{\operatorname{critwt}} 
\newcommand{\diag}{\operatorname{diag}}
\newcommand{\Disc}{\operatorname{Disc}}
\newcommand{\Div}{\operatorname{Div}}
\newcommand{\Dom}{\operatorname{Dom}}
\newcommand{\End}{\operatorname{End}}
\newcommand{\Fbar}{{\bar{F}}}
\newcommand{\Fix}{\operatorname{Fix}}
\newcommand{\Gal}{\operatorname{Gal}}
\newcommand{\GITQuot}{/\!/}
\newcommand{\GL}{\operatorname{GL}}
\newcommand{\GR}{\operatorname{\mathcal{G\!R}}}
\newcommand{\Hom}{\operatorname{Hom}}
\newcommand{\Index}{\operatorname{Index}}
\newcommand{\Image}{\operatorname{Image}}
\newcommand{\Isom}{\operatorname{Isom}}
\newcommand{\hhat}{{\hat h}}
\newcommand{\Ker}{{\operatorname{ker}}}
\newcommand{\Lift}{\operatorname{Lift}}
\newcommand{\limstar}{\lim\nolimits^*}
\newcommand{\limstarn}{\lim_{\hidewidth n\to\infty\hidewidth}{\!}^*{\,}}
\newcommand{\Mat}{\operatorname{Mat}}
\newcommand{\maxplus}{\operatornamewithlimits{\textup{max}^{\scriptscriptstyle+}}}
\newcommand{\MOD}[1]{~(\textup{mod}~#1)}
\newcommand{\Mor}{\operatorname{Mor}}
\newcommand{\Moduli}{\mathcal{M}}
\newcommand{\Norm}{{\operatorname{\mathsf{N}}}}
\newcommand{\notdivide}{\nmid}
\newcommand{\normalsubgroup}{\triangleleft}
\newcommand{\NS}{\operatorname{NS}}
\newcommand{\onto}{\twoheadrightarrow}
\newcommand{\ord}{\operatorname{ord}}
\newcommand{\Orbit}{\mathcal{O}}
\newcommand{\Pcase}[3]{\par\noindent\framebox{$\boldsymbol{\Pcal_{#1,#2}}$}\enspace\ignorespaces}
\newcommand{\Per}{\operatorname{Per}}
\newcommand{\Perp}{\operatorname{Perp}}
\newcommand{\PrePer}{\operatorname{PrePer}}
\newcommand{\PGL}{\operatorname{PGL}}
\newcommand{\Pic}{\operatorname{Pic}}
\newcommand{\Prob}{\operatorname{Prob}}
\newcommand{\Proj}{\operatorname{Proj}}
\newcommand{\Qbar}{{\bar{\QQ}}}
\newcommand{\rank}{\operatorname{rank}}
\newcommand{\Rat}{\operatorname{Rat}}
\newcommand{\Resultant}{\operatorname{Res}}
\renewcommand{\setminus}{\smallsetminus}
\newcommand{\sgn}{\operatorname{sgn}} 
\newcommand{\shafdim}{\operatorname{ShafDim}}
\newcommand{\SL}{\operatorname{SL}}
\newcommand{\Span}{\operatorname{Span}}
\newcommand{\Spec}{\operatorname{Spec}}
\renewcommand{\ss}{\textup{ss}}
\newcommand{\stab}{\textup{stab}}
\newcommand{\Stab}{\operatorname{Stab}}
\newcommand{\Support}{\operatorname{Supp}}
\newcommand{\TableLoopSpacing}{{\vrule height 15pt depth 10pt width 0pt}} 
\newcommand{\tors}{{\textup{tors}}}
\newcommand{\Trace}{\operatorname{Trace}}
\newcommand{\trianglebin}{\mathbin{\triangle}} 
\newcommand{\tr}{{\textup{tr}}} 
\newcommand{\UHP}{{\mathfrak{h}}}    
\newcommand{\val}{\operatorname{val}} 
\newcommand{\wt}{\operatorname{wt}} 
\newcommand{\<}{\langle}
\renewcommand{\>}{\rangle}

\newcommand{\pmodintext}[1]{~\textup{(mod}~#1\textup{)}}
\newcommand{\ds}{\displaystyle}
\newcommand{\longhookrightarrow}{\lhook\joinrel\longrightarrow}
\newcommand{\longonto}{\relbar\joinrel\twoheadrightarrow}
\newcommand{\SmallMatrix}[1]{%
  \left(\begin{smallmatrix} #1 \end{smallmatrix}\right)}

\newcommand{\SD}{\operatorname{\mathcal{SD}}}  
\newcommand{\MD}{\operatorname{\mathcal{MD}}}  


\begin{abstract}
Let $K$ be a number field, let $S$ be a finite set of places of $K$, and let $R_S$ be the ring of $S$-integers of $K$.  A $K$-morphism $f:\mathbb{P}^1_K\to\mathbb{P}^1_K$ has simple good reduction outside $S$ if it extends to an $R_S$-morphism $\mathbb{P}^1_{R_S}\to\mathbb{P}^1_{R_S}$.  A finite Galois invariant subset $X\subset\mathbb{P}^1_K(\bar{K})$ has good reduction outside $S$ if its closure in $\mathbb{P}^1_{R_S}$ is \'etale over $R_S$.  We study triples $(f,Y,X)$ with $X=Y\cup f(Y)$.  We prove that for a fixed $K$, $S$, and $d$, there are only finitely many $\text{PGL}_2(R_S)$-equivalence classes of triples with $\text{deg}(f)=d$ and $\sum_{P\in Y}e_f(P)\ge2d+1$ and $X$ having good reduction outside $S$.  We consider refined questions in which the weighted directed graph structure on $f:Y\to X$ is specified, and we give an exhaustive analysis for degree $2$ maps on $\mathbb{P}^1$ when $Y=X$.
\end{abstract}

\maketitle

\vspace*{-40pt}  

\tableofcontents

\section{Introduction}
\label{section:introduction}

Let~$K$ be a number field, let~$S$ be a finite set of places of~$K$
including all archimedean places, and let~$R_S$ be the ring
of~$S$-integers of~$K$.  We recall that an abelian variety~$A/K$ is
said to have \emph{good reduction outside~$S$} if there exists a
proper $R_S$-group scheme $\Acal/R_S$ whose generic fiber is
$K$-isomorphic to~$A/K$. Then we have the following famous conjecture
of Shafarevich, which was proven by Shafarevich in dimension~$1$ and
by Faltings in general.

\begin{theorem}
[Faltings~\cite{MR861971}] There are only finitely $K$-isomorphism
classes of abelian varieties~$A/K$ having good reduction outside~$S$.
\end{theorem}  

Our goal in this paper is to study an analogue of Shafarevich's
conjecture for dynamical systems on projective space.  The first
requirement is a definition of good reduction for self-maps
of~$\PP^N$, such as the
following~\cite{mortonsilverman:dynamicalunits}.

\begin{definition}
Let $f:\PP^N_{K}\to\PP^N_{K}$, be a non-constant
$K$-morphism. Then $f$ has (\emph{simple}) \emph{good reduction
  outside~$S$} if there exists an $R_S$-morphism
$\PP^N_{R_S}\to\PP^N_{R_S}$ whose generic fiber
is~$\PGL_{N+1}(K)$-conjugate to~$f$.
\end{definition}

If~$f$ has simple good reduction outside~$S$, and if $\f\in\PGL_{N+1}(R_S)$,
then it is clear that the conjugate map
\[
  f^\f := \f^{-1}\circ f\circ \f : \PP^N_K\longrightarrow\PP^N_K
\]
also has simple good reduction. But even modulo this equivalence, it is easy
to see that a dynamical analogue of Shafarevich's conjecture
using simple good reduction is false. For example, every map
$f:\PP_K^1\to\PP_K^1$ of the form
\[
  f(X,Y) = [X^d+a_1X^{d-1}Y+\cdots+a_dY^d,Y^d]
  \quad\text{with $a_i\in R_S$}
\]
has simple good reduction outside~$S$, and these maps represent
infinitely many~$\PGL_2(R_S)$-conjugacy classes.  And as noted
in~\cite[Example~4.1]{mortonsilverman:dynamicalunits}, there are also infinite
non-polynomial families such as
\[
  [aX^2+bXY+Y^2,X^2]\quad\text{with $a,b\in R_S$.}
\]
It is thus of interest to formulate alternative definitions of good
reduction for which a Shafarevich conjecture might hold in the
dynamical setting.  The literature contains several
papers~\cite{MR2999309,MR3293730,MR3273498,arxiv0603436} along these
lines.  We refer the reader to Section~\ref{section:earlier} for a
description of these earlier results and a comparison with the present
paper.

Our approach is to study pairs consisting of a map~$f$ and a set of
points~$Y\in\PP^N$ such that the map $f:Y\to f(Y)$ ``does not collapse''
when it is reduced modulo~$\gp$ for primes not in~$S$; see
Remark~\ref{remark:avtodyn} for a discussion of why this is a natural
analogue of the classical Shafarevich--Faltings result. To make this
precise, we need to define good reduction for sets of points.

\begin{definition}
Let $X\subset\PP^N(\Kbar)$ be a finite $\Gal(\Kbar/K)$-invariant
subset, say $X=\{P_1,\ldots,P_n\}$.  Then~$X$ has \emph{good reduction
  outside~$S$} if for every prime~$\gp\notin S$, and every prime~$\gP$
of~$K(P_1,\ldots,P_n)$ lying over~$\gp$, the reduction map\footnote{In
  scheme-theoretic terms, the set~$X$ is a reduced 0-dimensional
  $K$-subscheme of~$\PP^N_{K}$.  Let $\Xcal\subset\PP^N_{R_S}$ be the
  scheme-theoretic closure of~$X$.  Then~$X$ has good reduction
  outside~$S$ if~$\Xcal$ is \'etale over~$R_S$.}
\[
  X\longrightarrow \tilde X\bmod\gP\quad\text{is injective.}
\]
We observe that good reduction is preserved by the natural action of
$\PGL_{N+1}(R_S)$ on~$\PP^N(\Kbar)$.
\end{definition}

Our dynamical analogue of the Shafarevich--Faltings theorem is a
statement about triples~$(f,Y,X)$ consisting of a morphism~$f$ and
sets of points that have good reduction.  We restrict attention
to~$\PP^1$, since this is the setting for which we are currently able
to prove a strong Shafarevich-type theorem; but see
Section~\ref{section:generalizations} for a brief discussion of
possible extensions to~$\PP^N$ and why the naive generalization fails.

\begin{definition}
We define $\GR_d^1[n](K,S)$ to be the set of triples~$(f,Y,X)$,
where~$f:\PP^1_K\to\PP^1_K$ is a degree~$d$ morphism defined over~$K$
and $Y\subseteq X\subset\PP^1(\Kbar)$ are finite sets, satisfying the following
conditions:
\begin{parts}
  \Part{\textbullet}
  $X=Y\cup f(Y)$.
  \Part{\textbullet}
  $X$ is $\Gal(\Kbar/K)$-invariant.
  \Part{\textbullet}
  $\sum_{P\in Y} e_f(P)=n$, where~$e_f(P)$ is the ramification index of~$f$ at~$P$.
  \Part{\textbullet}
   $f$ and $X$ have good reduction outside $S$.
\end{parts}
We also define a potentially larger set $\widetilde{\GR}_d^1[n](K,S)$
by dropping the requirement that~$f$ has good reduction.  We observe
that if $Y=X$, then the points in~$X$ have finite~$f$-orbits, in which
case we say that~$(f,X,X)$ is a \emph{preperiodic triple}.
\end{definition}

There is a natural action of $\PGL_{N+1}(R_S)$ on $\GR_d^N[n](K,S)$
and on $\widetilde{\GR}_d^1[n](K,S)$ given by
\[
  \f\cdot(f,Y,X) := \bigl( f^\f, \f^{-1}(Y), \f^{-1}(X)\bigr).
\]
Our dynamical Shafarevich-type theorem for~$\PP^1$ says that if~$n$ is
sufficiently large, then $\widetilde{\GR}_d^1[n](K,S)$ has only
finitely many~$\PGL_{N+1}(R_S)$-orbits.

\begin{theorem}
[Dynamical Shafarevich Theorem for $\PP^1$]  
\label{theorem:shafconjP1}
Let $d\ge2$.
\begin{parts}
\Part{(a)}
Let~$K/\QQ$ be a number field, and let~$S$ be a finite set of places
of~$K$. Then for all $n\ge2d+1$, the set 
\[
  \smash[t]{\widetilde{\GR}}_d^1[n](K,S)/\PGL_{2}(R_S)~\text{is finite.}
\]
\Part{(b)}
Let $S$ be the set of rational primes dividing $(2d-2)!$.
Then
\[
  \GR_d^1[2d](\QQ,S)/\PGL_{2}(\ZZ_S)~\text{is infinite.}
\]
Indeed, there are infinitely many $\PGL_{2}(\ZZ_S)$-equivalence
classes of preperiodic triples~$(f,X,X)$ in $\GR_d^1[2d](\QQ,S)$.
\end{parts}
\end{theorem}
\begin{proof}
See Section~\ref{section:pfdynshafP1} for the proof of~(a), and see
Section~\ref{section:pfdynshafP1wt2d}, specifically
Proposition~\ref{proposition:fawt2d}, for the proof of~(b).
\end{proof}

In some sense, Theorem~\ref{theorem:shafconjP1} is the end of the
story for~$\PP^1$, since it says:
\begin{quotation}
``The Dynamical Shafarevich Conjecture is true for sets of weight at
  least~$2d+1$, but it is not true for sets of smaller weight.''
\end{quotation}
However, rather than merely specifying the total weight, we might
consider the weighted graph structure that $f:Y\to X$ imposes on~$X$,
where each point~$P\in Y$ is assigned an outgoing arrow $P\to f(P)$ of
weight~$e_f(P)$.  In dynamical parlance, we want to classify
triples~$(f,Y,X)$ according to their \emph{portrait
  structure}.\footnote{Portrait structures, especially on critical
  point orbits, are important tools in the study of complex dynamics
  on $\PP^1(\CC)$; see for example~\cite{arxiv1408.2118}.} 
The following (unweighted) example of a portrait illustrates the general idea:
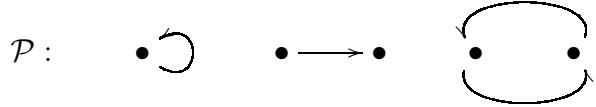
\[
  \xymatrix{
    \Pcal: & {\bullet} \ar@(dr,ur)[]_{} \\
  }
  \qquad
  \xymatrix{
    {\bullet} \ar[r]_{} & {\bullet}  \\
  }
  \qquad
  \xymatrix{
    {\bullet} \ar@(dl,dr)[r]_{}   & {\bullet}   \ar@(ur,ul)[l]_{}  \\
  }
\]
A model for this portrait~$\Pcal$ is a triple~$(f,Y,X)$
with $Y=\{P_1,P_2,P_4,P_5\}$ and  $X=\{P_1,P_2,P_3,P_4,P_5\}$ satisfying:
\begin{parts}
  \Part{\textbullet}
  $P_1$ is a fixed  point of $f$.
  \Part{\textbullet}
  $f(P_2)=P_3$.
  \Part{\textbullet}
  $P_4$ and $P_5$ form a periodic 2-cycle for $f$.
\end{parts}
If each point~$P\in\Pcal$ is assigned a weight~$\e(P)$, then we might
further require that $e_f(P)=\e(P)$, although there are other natural
possibilities. Indeed, we consider three ways to define good reduction
for dynamical systems and weighted portraits. We start with the
largest set and work our way down:

\begin{definition}
Let $\Pcal$ be a weighted portrait.
We define $\GR_d^1[\Pcal](K,S)$ to be the set of triples~$(f,Y,X)$,
where~$f:\PP^1_K\to\PP^1_K$ is a degree~$d$ morphism defined over~$K$
and $Y\subseteq X\subset\PP^1(\Kbar)$ are finite sets, satisfying the following
conditions:
\begin{parts}
  \Part{\textbullet}
  $X=Y\cup f(Y)$ and $f:Y\to X$ looks like~$\Pcal$ (ignoring the weights).
  \Part{\textbullet}
  $X$ is $\Gal(\Kbar/K)$-invariant.
  \Part{\textbullet}
   $f$ and $X$ have good reduction outside $S$.
\end{parts}
We then define three subsets of $\GR_d^1[\Pcal](K,S)$ by imposing the
following additional conditions on the triple~$(f,Y,X)$ that reflect
the weights assigned by~$\Pcal$:\footnote{We note that~$\star$-good
  reduction was first defined and studied by Petsche and
  Stout~\cite{MR3293730}, specifically for~$d=2$ and~$\Pcal$
  consisting of two fixed points or one $2$-cycle.}
\begin{align*}
  \GR_d^1[\Pcal]^\bullet(K,S) :\quad &\text{$e_f(P)\ge\e(P)$ for all $P\in Y$.} \\
  \GR_d^1[\Pcal]^\circ(K,S) :\quad &\text{$e_f(P)=\e(P)$ for all $P\in Y$.} \\
  \GR_d^1[\Pcal]^\star(K,S) :\quad &\text{$e_{\tilde f}(\tilde P\bmod\gp)=\e(P)$ for all $P\in Y$ and all $\gp\notin S$.} 
\end{align*}
\end{definition}

We refer the reader to Section~\ref{section:portraits} for rigorous
definitions of portraits, both weighted and unweighted, and their
models.  See also the companion paper~\cite{moduliportrait2017} in
which we construct parameter spaces and moduli spaces for dynamical
systems with portraits via geometric invariant theory and study some
of their geometric and arithmetic properties.

This leads to fundamental questions:

\begin{question}
\label{question:whichPareS}
For a given $d\ge2$, classify the portraits~$\Pcal$ having
the property that for all number fields~$K$ and all finite sets of
places~$S$, the set
\[
  \GR_d^N[\Pcal]^x(K,S)/\PGL_{2}(R_S)~\text{is finite, where $x\in\{\bullet,\circ,\star\}$.}
\]
If~$\Pcal$ has this property, then we say that~$\Pcal$ is an
\emph{$(x,d)$-Shafarevich portrait}, or that \emph{$(x,d)$-Shafarevich
  finiteness holds for~$\Pcal$}.
\end{question}

For example, Theorem~\ref{theorem:shafconjP1}(a) says that if the
total weight of the points in~$\Pcal$ is at least~$2d+1$, then
$(\bullet,d)$-Shafarevich finiteness holds for~$\Pcal$. This is quite
satisfactory. But the converse result, which is
Theorem~\ref{theorem:shafconjP1}(b), says only there exists at least
one portrait of total weight~$2d$ such that $(\bullet,d)$-Shafarevich
finiteness fails for~$\Pcal$. It says nothing about the full set of
such portraits.  And indeed, we will prove that among the
many portraits of total weight~$4$, $(\bullet,2)$-Shafarevich finiteness
holds for some and not for others!  Thus the answer to
Question~\ref{question:whichPareS} appears to be fairly subtle for
portraits of weight at most~$2d$.

In those cases that~$\GR_d^1[\Pcal]^x(K,S)$ is infinite, we might ask
for a more refined measure of its size. This is provided by looking at
its image in the moduli space $\Moduli_d^1$, where
$\Moduli_d^1:=\End_d^1\GITQuot\SL_{2}$ is the moduli space of
dynamical systems of degree~$d$ morphisms on~$\PP^1$.
(See~\cite{milnor:quadraticmaps, silverman:modulirationalmaps}
for the construction of~$\Moduli_d^1$,
and~\cite{MR2741188, MR2567424} for an analogous construction
for~$\PP^N$.)  This prompts the following definition.

\begin{definition}
Let $d\ge2$, let $x\in\{\bullet,\circ,\star\}$, and let~$\Pcal$ be a
portrait. The \emph{$(x,d)$-Shafarevich dimension of~$\Pcal$} is the
quantity
\[
  \shafdim_d^1[\Pcal]^x
  =\sup_{\substack{\text{$K$ a number field}\\\text{$S$ a finite set of places}\\}}
  \dim \overline{\operatorname{Image} \Bigl( \GR_d^1[\Pcal]^x(K,S) \to \Moduli_d^1 \Bigr)},
\]
where the overline denotes the Zariski closure.
\end{definition}

By definition, we have
\[
  \text{$\Pcal$ is a $(x,d)$-Shafarevich portrait}
  \quad\Longrightarrow\quad \shafdim_d^1[\Pcal]^x=0.
\]
A natural generalization of Question~\ref{question:whichPareS} is to
ask for a formula (or algorithm, or~\dots) for $\shafdim_d^1[\Pcal]^x$
as a function of~$\Pcal$.

In this paper we start to answer this refined question by performing
an exhaustive computation of $\shafdim_2^1[\Pcal]^x$ for preperiodic
portraits of weight up to~$4$, since
Theorem~\ref{theorem:shafconjP1}(a) says that the dimension is~$0$ for
portraits whose weight is strictly greater than~$4$.

To partially illustrate the complete results that are given in
Section~\ref{section:wt4deg2P1}, we refer the reader to
Table~\ref{table:wt4deg2P1examples}.  This table lists eight
preperiodic portraits of weight~$4$ that arise for degree~$2$ maps
of~$\PP^1$.  For six of them, the $(\bullet,2)$-Shafarevich finiteness
property holds, while for two of them it does not. It is not clear (to
this author) how to distinguish this dichotemy directly from the
geometry of the portraits, other than by performing a detailed
analysis.  It turns out that there are~$34$ possible portraits of
weight at most~$4$ for degree~$2$ maps of~$\PP^1$. See
Section~\ref{section:wt4deg2P1} for an analysis of all~$34$ portraits
and a computation of their various Shafarevich dimensions.

\begin{table}[t]   
\[
\begin{array}{|c|c|c|c|c|c|} \hline
  \Pcal & \text{$(1,2)$-Shafarevich Finiteness True?} \\ \hline\hline
  \xymatrix{ {\bullet} \ar[r]_{} & {\bullet}  \ar@(dr,ur)[]_{} } \quad
  \xymatrix{ {\bullet}  \ar@(dr,ur)[]_{} } \quad
  \xymatrix{ {\bullet}  \ar@(dr,ur)[]_{} } & \text{YES} \TableLoopSpacing\\ \hline
  \xymatrix{ {\bullet} \ar[r]_{} & {\bullet}  \ar@(dr,ur)[]_{} } \quad
  \xymatrix{ {\bullet} \ar[r]_{} & {\bullet}  \ar@(dr,ur)[]_{} } & \text{NO} \TableLoopSpacing \\ \hline
  \xymatrix{ {\bullet} \ar[r]_{} & {\bullet} \ar[r]_{} & {\bullet}  \ar@(dr,ur)[]_{} } \quad
  \xymatrix{  {\bullet}  \ar@(dr,ur)[]_{} }  & \text{YES} \TableLoopSpacing\\ \hline
  \xymatrix{ {\bullet} \ar[r] &  {\bullet} \ar[r] &  {\bullet} \ar[r] & {\bullet}  \ar@(dr,ur)[]_{} }
  & \text{YES}\TableLoopSpacing \\ \hline
\hline
  \xymatrix{ {\bullet}  \ar@(dr,ur)[]_{}} \quad
  \xymatrix{ {\bullet}  \ar@(dr,ur)[]_{}} \quad
  \xymatrix{
     {\bullet} \ar@(dl,dr)[r]_{}   & {\bullet}   \ar@(ur,ul)[l]_{} \\ } & \text{YES} \\ \hline
  \xymatrix{ {\bullet} \ar[r] & {\bullet}  \ar@(dr,ur)[]_{} } \quad
  \xymatrix{
     {\bullet} \ar@(dl,dr)[r]_{}   & {\bullet}   \ar@(ur,ul)[l]_{} \\
  }  & \text{NO} \\ \hline
  \xymatrix{ {\bullet}  \ar@(dr,ur)[]_{} } \quad
  \xymatrix{
     {\bullet} \ar[r] &  {\bullet} \ar@(dl,dr)[r]_{}   & {\bullet}   \ar@(ur,ul)[l]_{} \\
  }  & \text{YES} \\ \hline
  \xymatrix{
    {\bullet} \ar[r]_{} & {\bullet} \ar[r]_{} &
     {\bullet} \ar@(dl,dr)[r]_{}   & {\bullet}   \ar@(ur,ul)[l]_{} \\
  }  & \text{YES} \\ \hline
\end{array}
\]
\caption{Some weight $4$ portraits for degree $2$ maps}
\label{table:wt4deg2P1examples} 
\end{table}


We can also turn the question around by fixing~$\Pcal$ and
letting~$d\to\infty$.  We note that the Shafarevich dimension is never
more than $\dim\Mcal_d^1=2d-2$.

\begin{question}
\label{question:shafdisc}
For a given unweighted portrait $\Pcal$, what is the limiting behavior
of the \emph{Shafarevich discrepency}\footnote{If~$\Pcal$ has
  weights~$\e$, it is more natural quantity to consider the quantity
  $2d-2-\sum_{P\in Y}\bigl(\e(P)-1\bigr) - \shafdim_d^1[\Pcal]^x$ for
  $x\in\{\bullet,\circ,\star\}$.}
\[
  2d - 2 - \shafdim_d^1[\Pcal]\quad\text{as $d\to\infty$?}
\]
\end{question}

We note that Question~\ref{question:shafdisc} is quite interesting
even for $\Pcal=\emptyset$. We will show in
Proposition~\ref{proposition:simpledN1} that
\[
  d \le \shafdim_d^1[\emptyset] \le 2d-2.
\]
This gives the exact value for~$d=2$, a result that is also proven
in~\cite{MR3293730} using a a slightly different argument.

\begin{remark}
\label{remark:avtodyn}  
Returning to the case of abelian varieties for motivation and
inspiration, we note that an abelian variety is really a
pair~$(A,\Ocal)$ consisting of a variety and a marked point.
As noted by Petsche and Stout~\cite{MR3293730}, if we discard the
marked point, then Shafarevich finiteness is no longer true.
For example, there may be infinitely many $K$-isomorphism
classes of curves of genus~$1$ having good reduction outside~$S$.
Hence in order to prove Shafarevich finiteness for a collection
of geometric object (varieties, maps, etc.), it is very natural to
add level structure in the form of one or more points.
We also remark that if we add further level structure to an abelian
variety, for example specifying an $n$-torsion point~$Q$,
then an ostensibly stronger form of good reduction would require
that the points~$Q$ and~$\Ocal$ remain distinct modulo the primes
not in~$S$. But if we enlarge~$S$ so that~$n\in R_S^*$, then
the two forms of good reduction are actually identical
due to the standard result on injectivity of torsion under reduction;
cf.\ \cite[Theorem~C.1.4]{hindrysilverman:diophantinegeometry}
or~\cite[Appendix~II, Corollary~1]{MR0282985}. To make
the dynamical analogy complete, we  note that torsion
points are exactly the points of~$A$ that are preperiodic for the
doubling map.
\end{remark}

\section{Earlier results}
\label{section:earlier}
It has long been realized that dynamical Shafarevich finiteness does
not hold for morphisms $f:\PP^1\to\PP^1$ if the definition of good
reduction is simple good reduction;
cf.\ \cite[Example~4.1]{mortonsilverman:dynamicalunits}.  This has
led a number of authors to impose additional good reduction
conditions on~$f$ and to prove a variety of finiteness theorems. We
briefly mention a few of these results.

Closest in spirit to the present paper is work of Petsche and
Stout~\cite{MR3293730} in which they study good reduction of
degree~$2$ maps of~$\PP^1$.  They define (with similar notation) the
sets that we've denoted by~$\GR_d^1(K,S)[\emptyset]$ and they pose the
question of whether the maps in this set are Zariski dense in the
moduli space~$\Moduli_d^1$. They prove that this is true for~$d=2$,
which is a special case of our
Proposition~\ref{proposition:simpledN1}.  They also study maps with
$\star$-good reduction relative to various portraits, i.e., the
sets~$\GR_d^1[\Pcal]^\star$ defined earlier.  For example, they prove
that $\shafdim_2^1[\Pcal]^\star=1$ when~$\Pcal$ is a portrait
consisting of two unramified fixed points, and similarly when~$\Pcal$
is a portrait consisting of a single unramified~$2$-cycle.  (These are
the portraits labeled~$\Pcal_{2,3}$ and $\Pcal_{2,4}$ in
Table~\ref{table:wt3deg2P1}.)  We will show later
that~$\shafdim_2^1[\Pcal]^\circ=1$ and~$\shafdim_2^1[\Pcal]^\bullet=2$ for
these two portraits.  More generally, in
Section~\ref{section:wt4deg2P1} we compute the three Shafarevich
dimensions for the~34 preperiodic portraits of weight at most~4 for
degree~2 maps of~$\PP^1$.

Other approaches to a dynamical Shafarevich conjecture also consider
pairs~$(f,X)$ or triples~$(f,Y,X)$ of maps and points, but impose
different function-theoretic constraints. Thus Szpiro and
Tucker~\cite{arxiv0603436}, and later with
West~\cite{SzpiroFieldsTalk2017}, classify maps according to what
Szpiro characterizes as ``differential good reduction.''  For a given
map~$f:\PP^1\to\PP^1$, let $\Rcal(f)$ denote the set of ramified
points of~$f$ and let~$\Bcal(f)=f\bigl(\Rcal(f)\bigr)$ denote the set
of branch points.\footnote{In dynamical terminology,~$\Rcal(f)$ is the set of
\emph{critical points} and~$\Bcal(f)$ is the set of \emph{critical values.}}

\begin{definition}
The map $f$ has \emph{critical good reduction} outside~$S$ if each of
the sets~$\Rcal(f)$ and~$\Bcal(f)$ has good reduction outside~$S$.
The map $f$ has \emph{critical excellent reduction} outside~$S$ if the
union~$\Rcal(f)\cup\Bcal(f)$ has good reduction outside~$S$.
\end{definition}

Canci, Peruginelli, and Tossici~\cite{arxiv1103.3853} prove that~$f$
has critical good reduction if and only if~$f$ has simple good
reduction and the branch locus~$\Bcal(f)$ has good reduction.

\begin{theorem}
\label{theorem:STW}
\textup{(Szpiro--Tucker--West \cite{SzpiroFieldsTalk2017})}
Fix a number field~$K$, a finite set of places~$S$, and an
integer~$d\ge2$.  Then up to~$\PGL_2(K)$-conjugacy, there are only
finitely many degree~$d$ maps~$f:\PP^1_K\to\PP^1_K$ that are ramified
at~$3$ or more points and have critically good reduction outside~$S$.
\end{theorem}

Theorem~\ref{theorem:STW} of Szpiro, Tucker, and West fits into the
framework of our Theorem~\ref{theorem:shafconjP1}, since their
maps~$f$ correspond to triples
\[
  \bigl(f,\Rcal(f),\Rcal(f)\cup\Bcal(f)\bigr)\in
  \smash[t]{\widetilde{\GR}}_d^1[n](K,S),
\]
where  
\[
  n = \sum_{P\in\Rcal(f)} e_f(P) = \sum_{P\in\Rcal(f)} \bigl(e_f(P)-1\bigr) + \#\Rcal(f)
  = 2d-2+\#\Rcal(f).
\]
If we assume that~$\#\Rcal(f)\ge3$ as in Theorem~\ref{theorem:STW},
then $n\ge 2d+1$, so we see that Theorem~\ref{theorem:STW} follows
from Theorem~\ref{theorem:shafconjP1}(a).

The proof of Theorem~\ref{theorem:STW} in~\cite{SzpiroFieldsTalk2017}
uses a finiteness result for sets of points in~$\PP^1(K)$ having good
reduction outside~$S$, similar to our Lemmas~\ref{lemma:XP1finite}
and~\ref{lemma:fPGL2fX0}, which in turn rely on classical results of
Hermite and Minkowski together with the finiteness of solutions to the
$S$-unit equation. The other ingredient used by Szpior, Tucker, and
West in their proof of Theorem~\ref{theorem:STW} is a special case of
a theorem of Grothendieck that computes the tangent space of the
parameter scheme of morphisms. We remark 
that~\cite{SzpiroFieldsTalk2017,arxiv0603436} also deal with the case
of function fields, which can present additional complications.

The earlier paper~\cite{arxiv0603436} of Szpiro and Tucker proved a
result similar to Theorem~\ref{theorem:STW}, but with a two-sided
conjugation equivalence relation, i.e.,~$f_1$ and~$f_2$ are considered
equivalent if there are maps~$\f,\psi\in\PGL_2$ such that
$f_2=\psi\circ f_1\circ\f$.  This equivalence relation, while
interesting, is not well-suited for studying dynamics.

There is an article of Stout~\cite{MR3273498} in which he proves that
for a fixed rational map~$f$, there are only finitely many~$\Kbar/K$
twists of~$f$ having simple good reduction outside of~$S$.  And a
paper of Petsche~\cite{MR2999309} proves a Shafarevich finiteness
theorem for certain families of critically separable maps, which he
defines to be maps~$f$ of degree~$d\ge2$ such that for every prime not
in~$S$, the reduced map has~$2d-2$ distinct critical points. In other
words,~$\#\Rcal(f)=2d-2$ and~$\Rcal(f)$ has good reduction
outside~$S$. This is not enough to obtain finiteness, so Petsche
restricts to certain codimension~$3$ families in~$\Rat_d^1$ that are
modeled after Latt\`es maps, and he proves that the dynamical
Shafarevich conjecture holds for these families.

A number of authors have studied the resultant equation
$\Resultant(F,G)=c$, where the coefficients of~$F$ and~$G$ are viewed
as indeterminates~\cite{MR1234974,MR1117015,MR1243304}. Taking~$c$ to
be an~$S$-unit, this is clearly related to the question of simple good
reduction of the map~$f=[F,G]\in\End_d^1$.  Rephrasing the results in
our notation,\footnote{We have restricted to the case that
  $\deg(F)=\deg(G)$, although the cited papers do not require this.}
Evertse and Gy\H ory~\cite[Corollary 1]{MR1234974} prove that up
to~$\PGL_2(R_S)$-equivalence, there are only finitely
many~$f=[F,G]\in\End_d^1$ having the property that~$FG$ is square-free
and splits completely over~$K$. Alternatively, their conditions may be
phrased in terms of~$f$ as requiring that~$0$ and~$\infty$ are not
critical values of~$f$ and that the points in
$f^{-1}(0)\cup f^{-1}(\infty)$ are in~$\PP^1(K)$, and their conclusion
is that Shafarevich finiteness is true for this collection of maps. We
note that the condition that $f^{-1}(0)\cup
f^{-1}(\infty)\subset\PP^1(K)$ means, more-or-less, that the maps in
question correspond to $S$-integral points on a $2d$-to-$1$ finite
cover of an open subset of~$\End_d^1$.

Finally, we mention two topics that seem at least tangentially
related.  There are a number of papers that fix a map~$f$ and a
wandering point~$P$ and ask which portraits arise when one reduces the
orbit of~$P$ modulo various primes; see for
example~\cite{arxiv0903.1344,MR3349337}. And there are two articles of
Doyle~\cite{MR3487222,arxiv1603.08138} in which he classifies periodic
point portraits that are permitted for unicritical polynomials, i.e.,
polyomials of the form~$ax^d+b$.  These results could be useful in
studying the geometry and arithemtic of our portrait moduli spaces
studied in~\cite{moduliportrait2017}.

\section{Dynamical Shafarevich Finiteness Holds on $\PP^1$ for Weight ${}\ge2d+1$} 
\label{section:pfdynshafP1}
In this section we prove 
Theorem~\ref{theorem:shafconjP1}(a), namely we prove that the
dynamical Shafarevich finiteness holds for maps~$f$
of~$\PP^1$ and $f$-invariant sets~$X$ of weight at least~$2d+1$. The
first step is to show that there are only finitely many choices for
the set~$X$.

\begin{definition}
Let~$K$ be a number field, let~$S$ be a finite set of places including
all archimedean places, and let~$n\ge1$ be an integer. We
define~$\Xcal[n](K,S)$ to be the collection of subsets
$X\subset\PP^1(\Kbar)$ satisfying\textup:
\begin{parts}
\item[\textbullet]
  $\#X=n$.
\item[\textbullet]
  $X$ is $\Gal(\Kbar/K)$-invariant.
\item[\textbullet]
  $X$ has good reduction outside $S$.
\end{parts}
We note that if~$\f\in\PGL_2(R_S)$ and
$X=\{P_1,\ldots,P_n\}\in\Xcal[n](K,S)$, then
\begin{equation}
  \label{eqn:phiXphiP1phiPn}
  \f(X) := \bigl\{\f(P_1),\ldots,\f(P_n)\bigr\} \in \Xcal[n](K,S),
\end{equation}
so there is a natural action of~$\PGL_2(R_S)$ on~$\Xcal[n](K,S)$.
More generally, we use~\eqref{eqn:phiXphiP1phiPn} to define an action
of~$\PGL_2(\Kbar)$ on~$n$-tuples of points in~$\PP^1(\Kbar)$.
\end{definition}

The following lemma is well-known, but for lack of a suitable reference
and as a convenience to the reader, we include the proof.

\begin{lemma}
\label{lemma:XP1finite}
Fix a number field~$K$, a finite set of places~$S$ including all
archimedean places, and an integer~$n\ge3$. Then
\[
  \Xcal[n](K,S)/\PGL_2(R_S)
\]
is finite.
\end{lemma}

We start with a sublemma that will allow us to restrict attention to set of points
defined over a single field~$K$.

\begin{sublemma}
\label{lemma:fPGL2fX0}
Let~$K$ be a number field, let~$S$ be a finite set of places including
all archimedean places, and let~$n\ge3$ be an integer.
Then there is a constant~$C(K,S,n)$ such the map
\[
  \Xcal[n](K,S)/\PGL_2(R_S) \longrightarrow
  \bigl\{ X\subset\PP^1(\Kbar) : \#X=n \bigr\} /\PGL_2(\Kbar)
\]
is at most~$C(K,S,n)$-to-$1$.
\end{sublemma}
\begin{proof}
Let $X=\{P_1,\ldots,P_n\}\in\Xcal[n](K,S)$.  The fact that~$X$ is
Galois invariant implies that the field
\[
  K_X:= K(P_1,\ldots,P_n)
\]
is a Galois extension of degree dividing~$n!$.  Further, the good
reduction assumption on~$X$ implies that~$K_X/K$ is unramified
outside~$S$. It follows from the Hermite--Minkowski theorem
\cite[Section~III.2]{MR1697859} that there are only finitely many
possibilities for the field~$K_X$.\footnote{More precisely, our
  assumptions imply that for~$\gp\notin S$, we have $\ord_\gp
  \mathfrak{D}_{L/K}=0$, while for all primes~$\gp$ one has the
  standard estimate $\ord_\gp \mathfrak{D}_{L/K}\le[L:K]-1$.  This
  proves that~$\Norm_{L/K} \mathfrak{D}_{L/K}$ is bounded, and then
  for a fixed~$K$, Hermite--Minkowski says that there are only
  finitely many~$L$.}  It follows that the field
\begin{equation}
  \label{eqn:prodKX}
  K' := \prod_{X\in\Xcal[n](K,S)} K_X
\end{equation}
is a finite Galois extension of~$K$ that depends only on~$K$,~$S$,
and~$n$.

We now fix an $n$-tuple $X_0\in\Xcal[n](K,S)$, say
$X_0=\{Q_1,\ldots,Q_n\}$, and consider the set of $n$-tuples
in~$\Xcal[n](K,S)$ that are~$\PGL_2(\Kbar)$-equivalent to~$X_0$. Our
goal is to prove that the set
\[
  \PGL_2(K,S,X_0) :=
  \bigl\{ \f\in\PGL_2(\Kbar) : \f(X_0)\in\Xcal[n](K,S) \bigr\}
\]
has the property that~$\PGL_2(K,S,X_0)/\PGL_2(R_S)$ is finite and has
order bounded solely in terms of~$K$,~$S$, and~$n$.

Our first observation is that if $\f\in\PGL_2(K,S,X_0)$, then in
particular we have $Q_i\in\PP^1(K')$ and $\f(Q_i)\in\PP^1(K')$ for all
$1\le i\le n$, where~$K'$ is the field~\eqref{eqn:prodKX}. A
fractional linear transformation is determined by its values at three
points, so our assumption that~$n\ge3$ tells us
that~$\f\in\PGL_2(K')$, i.e., every~$\f\in\PGL_2(K,S,X_0)$ is defined
over the fiinite extension~$K'$ of~$K$, where~$K'$ does not depend
on~$X_0$.

Next let~$S'$ be the places of~$K'$ lying over~$S$. The good reduction
assumption on~$X_0$ and~$\f(X_0)$ implies that $Q_1,\ldots,Q_n$ remain
distinct modulo all primes~$\gP$ of~$L$ with~$\gP\notin S'$, and
similarly for $\f(Q_1),\ldots,\f(Q_n)$. Since~$n\ge3$, we can apply
the following elementary result to conclude that~$\f$ has simple good
reduction at~$\gP$, and since this holds for all~$\gP\notin S'$, we
see that~$\f\in\PGL_2(R_{S'})$.

\begin{sublemma}
\label{sublemma:3ptsmodp} 
Let $R$ be a discrete valuation ring with maximal ideal $\gP$ and
fraction field~$K$. Let
$P_1,P_2,P_3\in\PP^1(K)$ be points whose reductions modulo~$\gP$ are
distinct, and let $Q_1,Q_2,Q_3\in\PP^1(K)$ also be points with
distinct mod~$\gP$ reductions. Let $\f\in\PGL_2(K)$ be the unique linear
fractional transformation satisfying $\f(P_i)=Q_i$ for $1\le i\le 3$.
Then $\f\in\PGL_2(R)$, i.e.,~$\f$ has good reduction modulo~$\gP$.
\end{sublemma}
\begin{proof}
The fact that the reductions of~$P_1,P_2,P_3$ are distinct means that
we can find a linear fractional transformation~$\psi\in\PGL_2(R)$
satisfying $\psi(P_1)=0$, $\psi(P_2)=1$, $\psi(P_3)=\infty$.
Similarly, we can find a~$\l\in\PGL_2(R)$ satisfying
$\l(Q_1)=0$, $\l(Q_2)=1$, $\l(Q_3)=\infty$. Then
$\l\circ\f\circ\psi^{-1}$ fixes~$0$,~$1$, and~$\infty$, so it is the identity
map. Hence $\f=\l^{-1}\circ\psi\in\PGL_2(R)$.
\end{proof}

We next observe that if $\f\in\PGL_2(K,S,X_0)$, then by definition and
from what we proved earlier, both of the sets~$X_0$ and~$\f(X_0)$ are
composed of points in~$\PP^1(K')$ and both are~$\Gal(K'/K)$-invariant.
Hence for any~$\s\in\Gal(K'/K)$, we find that
\[
  \f(X_0)=\bigl(\f(X_0)\bigr)^\s=\f^\s(X_0^\s)=\f^\s(X_0).
\]
Thus $\f^{-1}\circ\f^\s:X_0\to X_0$, i.e., the map~$\f^{-1}\circ\f^\s$ is a
permutation of the set~$X_0$. We thus obtain a map
\[
\begin{array}{ccc}
  \PGL_2(K,S,X_0) &\longrightarrow& \operatorname{Map}_{\textsf{Set}}\bigl(\Gal(K'/K),\Scal_{X_0}\bigr), \\
  \f &\longmapsto& \bigl( \s \longmapsto \f^{-1}\circ\f^\s \bigr),
\end{array}
\]
where $\Scal_{X_0}$ denotes the group of permutations of the
set~$X_0$.  (The map $\s\mapsto\f^{-1}\circ\f^\s$ is actually some
sort of cocycle, but that is irrelevant for our purposes.)
Since~$\Gal(K'/K)$ and~$\Scal_{X_0}$ are both finite and have order
bounded in terms of~$K$,~$S$, and~$n$, it suffices to fix some
$\f_0\in\PGL_2(K,S,X_0)$ and to bound the number of~$\PGL_2(R_S)$
equivalence classes of maps $\f\in\PGL_2(K,S,X_0)$ that have the same
image in
$\operatorname{Map}_{\textsf{Set}}\bigl(\Gal(K'/K),\Scal_{X_0}\bigr)$.
This means that for all $\s\in\Gal(K'/K)$, the maps
$\f^{-1}\circ\f^\s=\f_0^{-1}\circ\f_0^\s$ have the same effect
on~$X_0$; and since~$\#X_0=n\ge3$ and linear fractional
transformations are determined by their values on three points, it
follows that $\f^{-1}\circ\f^\s=\f_0^{-1}\circ\f_0^\s$ as elements
of~$\PGL_2(K')$. Thus
\[
  \f\circ\f_0^{-1} = \f^\s\circ(\f_0^\s)^{-1} = (\f\circ\f_0^{-1})^\s
  \quad\text{for all $\s\in\Gal(K'/K)$.}
\]
Hence $\f\circ\f_0^{-1}\in\PGL_2(K)$. But we also know that~$\f_0$ and~$\f$
are in~$\PGL_2(R_{S'})$, so
\[
  \f\circ\f_0^{-1}\in\PGL_2(K)\cap\PGL_2(R_{S'}).
\]

It remains to show that 
\begin{equation}
  \label{eqn:PGL2KPGL2R}
  \PGL_2(K)\cap\PGL_2(R_{S'}) = \PGL_2(R_S),
\end{equation}
since that will show that up to composition with  elements
of $\PGL_2(R_S)$, there are only finitely many choices for~$\f$.  In
order to prove~\eqref{eqn:PGL2KPGL2R}, we start with some
$\psi\in\PGL_2(K)\cap\PGL_2(R_{S'})$. Then for each prime~$\gp\notin S$,
we need to show that~$\psi$ has good reduction at~$\gp$. We write~$\psi$
in normalized form as
\begin{multline}
  \label{eqn:psinormalized}
  \psi(X,Y)=[aX+bY,cX+dY] \quad\text{with}\quad a,b,c,d\in K\\
  \text{and}\quad
  \min\bigl\{ \ord_\gp(a),\ord_\gp(b),\ord_\gp(c),\ord_\gp(d) \bigr\} = 0,
\end{multline}
i.e.,~$a,b,c,d$ are all~$\gp$-integral, and at least one of them is
a~$\gp$-unit. Now let~$\gP$ be a prime of~$K'$ lying above~$\gp$. We
are given that~$\psi$ has good reduction at~$\gP$, which means that if
we choose a $\gP$-normalized equation for~$\psi$, its reduction
modulo~$\gP$ has good reduction.  But~\eqref{eqn:psinormalized} is already
normalized for~$\gP$, since $\ord_\gP=e(\gP/\gp)\ord_\gp$. Hence
\[
  \text{$ad-bc$ is a $\gP$-adic unit.}
\]
But $ad-bc\in K$, so~$ad-bc$ is a~$\gp$-adic unit, and hence~$\psi$
has good reduction at~$\gp$. This holds for all~$\gp\notin S$, which
completes the proof that~$\psi\in\PGL_2(R_S)$, and thus completes
the proof of Sublemma~\ref{lemma:fPGL2fX0}.
\end{proof}

\begin{proof}[Proof of Lemma~$\ref{lemma:XP1finite}$]
Let $L/K$ be a finite Galois extension, and let~$T$ be a finite of
places of~$L$ whose restriction to~$K$ contains~$S$. Then we get a
natural map
\begin{equation}
  \label{eqn:XKStoXLT}
  \Xcal[n](K,S)/\PGL_2(R_S) \longrightarrow \Xcal[n](L,T)/\PGL_2(R_T),
\end{equation}
since if~$X\subset\PP^1(\Kbar)$ is~$\Gal(\Kbar/K)$ invariant and has
good reduction outside~$S$, it is clear that~$X$ is
also~$\Gal(\bar{L}/L)$ invariant and has good reduction
outside~$T$. However, what  is not clear \emph{a priori} is that the
map~\eqref{eqn:XKStoXLT} is finite-to-one, since~$\PGL_2(R_T)$ may be
larger than~$\PGL_2(R_S)$.

However Sublemma~\ref{lemma:fPGL2fX0} says not only that the map
\[
  \Xcal[n](K,S)/\PGL_2(R_S)
  \longrightarrow   \bigl\{ X\subset\PP^1(\Kbar) : \#X=n \bigr\} /\PGL_2(\Kbar)
\]
is finite-to-one, but it also says that the number of elements in
each $\PGL_2(\Kbar)$-equivalence class of $\Xcal[n](K,S)/\PGL_2(R_S)$
is bounded solely in terms of~$K$,~$S$, and~$n$. Hence
using~\eqref{eqn:XKStoXLT}, it suffices to prove
Lemma~\ref{lemma:XP1finite} for any such~$L$ and~$T$.

As shown in the proof of Sublemma~\ref{lemma:fPGL2fX0}, there is a finite
extension~$K'/K$ such that every~$X\in\Xcal[n](K,S)$ is an $n$-tuple
of points in~$\PP^1(K')$. We then let~$S'$ be a finite set of places
of~$K'$ such that~$S'$ restricted to~$K$ contains~$S$ and such
that~$R_{S'}$ is a PID. Replacing~$K$ and~$S$ with~$K'$ and~$S'$, we
are reduced to studying the $\PGL_2(R_S)$-equivalence classes of the
set of~$X\in\Xcal[n](K,S)$ such that
\[
  X=\{P_1,\ldots,P_n\}\quad\text{with}\quad P_1,\ldots,P_n\in\PP^1(K),
\]
with the further condition that~$R_S$ is a PID.  This allows us to
choose normalized coordinates for the points in~$X$, say
\[
  P_i = [a_i,b_i] \quad\text{with $a_i,b_i\in R_S$ and
    $\gcd\nolimits_{R_S}(a_i,b_i)=1$.}
\]
The good reduction assumption says that~$P_1,\ldots,P_n$ are distinct
modulo all primes not in~$S$, which given our normalization of the
coordinates of the~$P_i$, is equivalent to the statement that
\[
  a_ib_j-a_jb_i\in R_S^*\quad\text{for all $1\le i<j\le n$.}
\]
This means that we can find a linear fractional transformation
$\f\in\PGL_2(R_S)$ that moves the first three points in our list to
the points 
\[
  \f(P_1)=[1,0],\quad \f(P_2)=[0,1],\quad \f(P_3)=[1,1].
\]
Replacing~$X$ by~$\f(X)$, the remaining points in~$X$ are~$S$-integral
points of the scheme
\begin{equation}
  \label{eqn:P1RS100111}
  \PP^1_{R_S}\setminus \bigl\{[1,0],[0,1],[1,1]\bigr\},
\end{equation}
and it is well-known that there are only finitely many such points.
More precisely, a normalized point~$P=[a,b]$ is an $S$-integral point
of the scheme~\eqref{eqn:P1RS100111} if and only if~$a$,~$b$,
and~$a-b$ are $S$-units. But this implies that
$(\frac{a}{a-b},\frac{b}{b-a})$ is a solution to the $S$-unit equation
$U+V=1$, and hence that there are only finitely many values for each
of~$\frac{a}{a-b}$
and~$\frac{b}{b-a}$~\cite[IX.4.1]{MR2514094}. Further, each $S$-unit
solution~$(u,v)$ to $u+v=1$ gives one point $P=[a,b]=[u,-v]$. This
concludes the proof that there are only finitely
many~$\PGL_2(R_S)$-equivalence classes of sets~$X$ having~$n$ elements
and good reduction outside~$S$.
\end{proof}  

The following geometric result is also undoubtedly well-known, but for
lack of a suitable referece and the convenience of the reader, we
include the short proof.

\begin{lemma}
\label{lemma:interpratfnc}
Let $K$ be a field, and let $f,g:\PP_K^1\to\PP_K^1$ be rational maps of
degree~$d\ge1$. Suppose that
\[
  \sum_{\substack{P\in\PP^1(K)\\ f(P)=g(P)\\}} \min\bigl\{ e_f(P),e_g(P) \bigr\} \ge 2d+1.
\]
Then $f=g$. 
\end{lemma}
\begin{proof}
We may assume that~$K$ is algebraically closed.  We fix a basepoint
$P_0\in\PP^1(K)$, and we take
\[
  H_1 = \{P_0\}\times\PP^1
  \quad\text{and}\quad
  H_2=\PP^1\times\{P_0\}
\]
as generators for $\Pic(\PP^1\times\PP^1)\cong\ZZ\oplus\ZZ$. We
consider the divisors
\begin{align*}
  \D&=\bigl\{ (P,P) : P\in\PP^1(K)\bigr\} \in \Div(\PP^1\times\PP^1), \\
  \G_{f,g}&= (f\times g)_*\D \in \Div(\PP^1\times\PP^1).
\end{align*}
We write $|\D|$ and $|\G_{f,g}|$ for the supports of~$\D$
and~$\G_{f,g}$, respectively, and we note that these supports are
irreducible, since they are the images of~$\PP^1$ under, respectively,
the diagonal map and the map~$f\times g$.

We use the push-pull formula to compute the global intersection
\[
  \G_{f,g}\cdot H_1 = (f\times g)_*(\D) \cdot H_1
  = \D\cdot (f\times g)^*(H_1)
  = \D\cdot (f^*(P_0)\times\PP^1)
  = d.
\]
Siimlarly, we have $\G_{f,g}\cdot H_2=d$. Hence
\[
  \G_{f,g} = d H_1 + d H_2\quad\text{in $\Pic(\PP^1\times\PP^1)$.}
\]
This allows us to compute
\begin{equation}
  \label{eqn:GfgD2d}
  \G_{f,g}\cdot \D = d H_1\cdot\D+d H_2\cdot\D = 2d.
\end{equation}

Choose some~$P\in\PP^1(K)$ satisfying~$f(P)=g(P)$, and let~$z$ be a
local uniformizer at~$P$.  We may assume
that~$z\bigl(f(P)\bigr)\ne\infty$, since otherwise we can replace~$z$
by~$z/(z-1)$.  By assumption we have $c:=f(P)=g(P)$, so locally
near~$P$ the functions~$f$ and~$g$ look like
\[
  f(z) = c + a z^{e_f(P)} + \text{(h.o.t)},\quad
  g(z) = c + b z^{e_g(P)} + \text{(h.o.t)},
\]
for some nonzero~$a$ and~$b$. This allows us to
estimate the following local intersection index:
\begin{align}
  \label{eqn:localPiei}
  \bigl( (f\times g)_*\D\cdot\D \bigr)_{(f(P),g(P))}
  &= \dim_K \frac{K[\![x,y,z]\!]}{\bigl(x-f(z),y-g(z),x-y\bigr)} \notag\\*
  &= \dim_K \frac{K[\![z]\!]}{\bigl(f(z)-g(z)\bigr)} \notag\\*
  &\ge \min\bigl\{ e_f(P),e_g(P) \bigr\}.
\end{align}

Suppose that $|\G_{f,g}|\cap|\D|$ is finite.  Then we can
calculate $\G_{f,g}\cdot\D$ as a sum of local intersections.
Combined with~\eqref{eqn:localPiei}, this yields
\begin{align*}
  2d &= \G_{f,g}\cdot\D \quad\text{from \eqref{eqn:GfgD2d},}\\*
  &= \sum_{Q\in\PP^1(K)} \bigl( (f\times g)_*\D\cdot\D \bigr)_{(Q,Q)}
  \quad\text{since $|\G_{f,g}|\cap|\D|$ is finite,} \\
  &= \sum_{\substack{Q\in\PP^1(K)~\text{such that}\\ \exists P\in\PP^1(K)~\text{with}~f(P)=g(P)=Q\\}}
       \bigl( (f\times g)_*\D\cdot\D \bigr)_{(Q,Q)} \\
  &\ge \sum_{\substack{P\in\PP^1(K)\\f(P)=g(P)\\}} \min\bigl\{ e_f(P),e_g(P) \bigr\}
  \quad\text{from \eqref{eqn:localPiei},} \\*
  &\ge 2d+1 \quad\text{by assumption.}
\end{align*}
Thus the assumption that $|\G_{f,g}|\cap|\D|$ is finite leads to a
contradiction. It follows that~$|\D|$ and~$|\G_{f,g}|$ have a a common
positive dimensional component. But as noted earlier, both~$|\D|$
and~$|\G_{f,g}|$ are irreducible curves, and hence
$|\D|=|\G_{f,g}|$. Thus~$f$ and~$g$ take on the same value at every
point of~$\PP^1(K)$, and therefore~$f=g$, which completes the proof of
Lemma~\ref{lemma:interpratfnc}.
\end{proof}  

We now have the tools needed to prove dynamical Shafarevich finiteness
for~$\PP^1$.

\begin{proof}[Proof of Theorem~\textup{\ref{theorem:shafconjP1}(a)}]
Our goal is to prove that
\[
  \smash[t]{\widetilde{\GR}}_d^1[n](K,S)/\PGL_{2}(R_S)~\text{is finite.}
\]
Let $(f,Y,X)\in\smash[t]{\widetilde{\GR}}_d^1[n](K,S)$, and let $\ell=\#X$. We note that
\[
  2d+1 \le n = \sum_{P\in Y} e_f(P) \le d\cdot\#Y \le d\cdot\#X = d\ell,
\]
so $\ell\ge3$. Further, the set~$X$
is~$\Gal(\Kbar/K)$-invariant and has good reduction outside
of~$S$. Lemma~\ref{lemma:XP1finite} tells us that up
to~$\PGL_2(R_S)$-equivalence, there are only finitely many
possibilities for~$X$. So without loss of generality, we may assume
that $X=\{P_1,\ldots,P_\ell\}$ is fixed.

The set~$Y$ is subset of~$X$, so there are only finitely many choices
for~$Y$. Relabeling the points in~$X$, we may thus also assume
that $Y=\{P_1,\ldots,P_m\}$ is fixed.

By definition, the map~$f$ satisfies $X=f(Y)\cup Y$, so in particular,
$f(Y)\subset X$.  Thus~$f$ induces a map
\[
  \nu_f:\{1,\ldots,m\}\to\{1,\ldots,\ell\}
  \quad\text{characterized by}\quad
  f(P_i) = P_{\nu_f(i)}.
\]
There are only~$m^\ell$ maps~$\nu$ from the set~$\{1,\ldots,m\}$ to
the set~$\{1,\ldots,\ell\}$, so again without loss of generality, we
may fix one map~$\nu$ and restrict attention to maps~$f$ satisfying
$\nu_f=\nu$. This means that the value of~$f$ is specified at each
of the points~$P_1,\ldots,P_m$ in~$Y$.

We define a the map
\[
  \smash[t]{\widetilde{\GR}}_d^1[n](K,S) \longrightarrow \ZZ^m,\quad
  (f,X) \longmapsto \bigl( e_f(P_1),\ldots,e_f(P_m) \bigr).
\]
Since~$e_f(P)$ is an integer between~$1$ and~$d$, there are only
finitely many possibilities for the image. We may thus restrict
attention to triples~$(f,Y,X)$ such that the ramification indices of~$f$
at the points in~$Y$ are fixed.

But now any two triples~$(f,Y,X)$ and~$(g,Y,X)$ have the same values and the
same ramification indices at the points in~$Y$, and by assumption the
sum of those ramification indices is at least~$2d+1$, so
Lemma~\ref{lemma:interpratfnc} tells us that~$f=g$. This completes the
proof that $\smash[t]{\widetilde{\GR}}_d^1[n](K,S)$ contains only
finitely many~$\PGL_2(R_S)$-equivalence classes of triples~$(f,Y,X)$.
\end{proof}  

\section{Dynamical Shafarevich Finiteness Fails on $\PP^1$ for Weight ${}\le2d$} 
\label{section:pfdynshafP1wt2d}

In this section we prove Theorem~\ref{theorem:shafconjP1}(b). More
precisely, we prove that the dynamical Shafarevich finiteness is false
for maps~$f:\PP^1\to\PP^1$ and $f$-invariant sets~$X$ containing~$2d$
points.  We do this by analyzing a particular family of maps.

\begin{proposition}
\label{proposition:fawt2d}
Let $d\ge2$, let~$K/\QQ$ be a number field, and let $S$ be the set of
primes of~$K$ dividing $(2d-2)!$. For each $a\in\Kbar^*$, let~$f_a(x)$ be
the map
\[
  f_a(x) = \frac{ax(x-1)(x-2)\cdots(x-d+1)}{(x+1)(x+2)\cdots(x+d-1)}\in\End_d^1,
\]
and let $X\subset\PP^1$ be the set
\[
  X = \{0,1,2,\ldots,d-1\} \cup \{-1,-2,\ldots,-(d-1)\} \cup \{\infty\}.
\]
\vspace{-10pt}
\begin{parts}
  \Part{(a)}
  For all $a\in R_S^*$, we have
  \[
    (f_a,X,X)\in\GR_d^1[2d](K,S).
  \]
  \Part{(b)}
  For a given~$a\in\Kbar^*$, there are only finitely
  many~$b\in\Kbar^*$ such that~$f_b$ is~$\PGL_2(\Kbar)$-conjugate
  to~$f_a$.
  \Part{(c)}
  $\#\GR_d^1[2d](K,S)/\PGL_2(R_S)=\infty$.
\end{parts}
\end{proposition}
\begin{proof} 
(a) The resultant of~$f_a$ is
\[
  \Resultant(f_a) = a^d \prod_{i=0}^{d-1}\prod_{j=1}^{d-1} (i+j).
\]
In particular, if~$a\in R_S^*$, then our choice of~$S$ implies that
$\Resultant(f_a)\in R_S^*$, so the map~$f_a$ has simple good reduction
outside~$S$. We also observe that our choice of~$S$ implies that the set $X$
has good reduction outside~$S$, and from the formula for~$f_a$ we see
that $f_a(X)=\{0,\infty\}\subset X$. For example, the case $d=4$ looks like
\[
  \xymatrix{
    {1} \ar[dr]_{} \\
    {2} \ar[r]_{} & {0} \ar@(dr,ur)[]_{} \\
    {3} \ar[ur]_{} \\
  }
  \qquad
  \xymatrix{
    {-1} \ar[dr]_{} \\
    {-2} \ar[r]_{} & {\infty} \ar@(dr,ur)[]_{} \\
    {-3} \ar[ur]_{} \\
  }
  \qquad\raisebox{-40pt}{with $S=\{2,3,5\}$.}
\]
Since $\#X=2d$, this completes the proof that $(f_a,X,X)\in\GR_d^1[2d](K,S)$.
\par\noindent(b)\enspace
We consider the $\Kbar$-valued points of the morphism 
\begin{equation}
  \label{eqn:CtoMdatofa}
  \Kbar^*\longrightarrow \Moduli^1_d(\Kbar)=\End^1_d(\Kbar)/\PGL_2(\Kbar),\quad
  a \longmapsto [f_a].
\end{equation}
We claim that the map~\eqref{eqn:CtoMdatofa} is non-constant. To see
this, we note that~$0$ is a fixed point of~$f_a$, and that the
multiplier of~$f_a$ at~$0$ is
\[
\l(f_a,0):=f_a'(0)=(-1)^{d-1}a.
\]
But for any rational map~$f\in\End^1_d$, the set of fixed point
multipliers $\{\l(f,P):P\in\Fix(f)\}$ is a $\PGL_2$-conjugation
invariant~\cite[Proposition~1.9]{MR2316407}.  So
if~\eqref{eqn:CtoMdatofa} were constant, there would be a single
map~$g\in\End^1_d(\Kbar)$ with the property that for
every~$a\in\Kbar^*$, the map~$f_a$ is~$\PGL_2(\Kbar)$-conjugate
to~$g$.  In particular, for every~$a\in\Kbar^*$, the multiplier
$(-1)^{d-1}a=\l(f_a,0)$ would be one of the finitely many fixed-point
multipliers of~$g$. This contradiction completes the proof of~(b)
\par\noindent(c)\enspace It follows from~(a) and~(b) that
$\bigl\{(f_a,X,X):a\in R_S^*\bigr\}$ is contained
in~$\GR_d^1[2d](K,S)$ and that it contains infinitely many
distinct~$\PGL_2(R_S)$-conjugacy classes.
\end{proof}

\section{How Large is the Set of Maps Having Simple Good Reduction?}
\label{section:howbigsimple}

As noted in the introduction, it would be very interesting to know the
behavior of the ``Shafarevich discrepency,''
\[
   2d - 2 - \shafdim_d^N[\Pcal]\quad\text{as $d\to\infty$,}
\]
even for the case $\Pcal=\emptyset$. It has long been noted that monic
polynomial maps on~$\PP^1$ have everywhere simple good reduction. This
gives a set of such maps in $\Moduli_d^1$ whose Zarkiski closure has
dimension~$d-1$. With a little work, we can increase this dimension
by~$1$.

\begin{proposition}
\label{proposition:simpledN1}
For all $d\ge2$ we have
\[
   \shafdim_d^1[\emptyset] \ge d .
\]
\end{proposition}
\begin{proof}
We fix a number field $K$ and set of places $S$ so that $R_S^*$ is infinite.
For $\bfa=(a_0,a_2,\ldots,a_{d-2},a_d)$ we define a rational map
\[
  f_\bfa(x) := \frac{a_0x^d+a_1x^{d-1}+\cdots+a_{d-2}x^2+x+a_d}{x}.
\]
Note that $\bfa$ is a $d$-tuple, since there is no~$a_{d-1}$ term.  We
have $\Resultant(f_\bfa)=a_0^{d-1}a_d$, so~$f_\bfa$ has simple good
reduction for all $\bfa\in R_S^*\times R_S^{d-2} \times R_S^*$.  This
set of~$\bfa$ is Zariski dense in $\AA^d$, so it remains to show that
the map $\AA^d\to\Moduli_d^1$ given by $\bfa\mapsto\<f_\bfa\>$ is
generically injective (or at worst finite-to-one).

Suppose that $\f\in\PGL_2(\Kbar)$ has the property that
$f_\bfa^\f=f_\bfb$.  We start with the case $d\ge3$.  Then~$f_\bfa$ is
ramified at the fixed point at~$\infty$, since
$e_{f_\bfa}(\infty)=d-1$, and similarly
for~$f_\bfb$. Generically,~$\infty$ will be the only ramified fixed
point of~$f_\bfa$ and~$f_\bfb$, so~$\f(\infty)=\infty$.  Next we use
the fact that $f_\bfa^{-1}(\infty)=f_\bfb^{-1}(\infty)=\{\infty,0\}$
to conclude that~$\f(0)=0$. Thus~$\f(x)=\a x$. The coefficient of~$x$
in the numerator of $f_\bfa^\f(x)=\a^{-1}f_\bfz(\a x)$ is $\a^{-1}$,
so comparing with $f_\bfb$, we conclude that $\a=1$. This concludes
the proof for $d\ge3$.

For $d=2$, we use the Milnor isomorphism
$s:\Moduli^1_2\xrightarrow{\;\sim\;}\AA^2$;
see~\cite[Theorem~4.5.6]{MR2316407}.  The map~$f_\bfa(x)=(ax^2+x+b)/x$
has Milnor coordinates
\begin{multline*}
  s\left( \frac{ax^2+x+b}{x}\right) =
  \left( \frac{4a^2b-2ab-a+b}{ab},\right. \\*
  \left. \frac{4a^3b-4a^2b-a^2+5ab-2b-1}{ab} \right).
\end{multline*}
We used Magma~\cite{MR1484478} to verify that these two rational
functions are algebraically independent in $K(a,b)$. Hence under our
assumption that~$\#R_S^*=\infty$, we see that $\{s(a,b):a,b\in
R_S^*\}$ is Zariski dense in~$\AA^2$.
\end{proof}


\section{Abstract Portraits and Models for Portraits}
\label{section:portraits}
In this section we briefly construct a category of portraits and use
it to describe dynamical systems that model a given portrait.
See~\cite{moduliportrait2017} for further development and the
construction of parameter and moduli spaces for dynamical systems with
portraits.

\begin{definition}
An (\emph{abstract}) \emph{weighted portrait} is a
4-tuple~$\Pcal=(\Wcal,\Vcal,\Phi,\e)$, where
\begin{parts}
  \Part{\textbullet}
  $\Wcal\subseteq \Vcal$ are a finite sets (of vertices);
  \Part{\textbullet}
  $\Phi:\Wcal\to\Vcal$ is a map (which specifies directed edges).
  \Part{\textbullet}
  $\Vcal = \Wcal\cup\Phi(\Wcal)$.
  \Part{\textbullet}
  $\e:\Wcal\to\NN$ is a map (assigning weights  to vertices).
\end{parts}
The \emph{weight of~$\Pcal$} is the total weight
\[
  \wt(\Pcal) := \sum_{w\in\Wcal} \e(w).
\]

We say that the portrait is \emph{unweighted} if $\e(w)=1$ for every
$w\in\Wcal$, or equivalently if $\wt(\Pcal)\#\Wcal$, in which case we
sometimes write~$\Pcal=(\Wcal,\Vcal,\Phi)$.  We say that the portrait
is \emph{preperiodic} if $\Wcal=\Vcal$.
\end{definition}

We now explain how a self-map of $\PP^1$ can be used to model a
portrait.

\begin{definition}
Let $\Pcal=(\Wcal,\Vcal,\Phi,\e)$ be a portrait. A \emph{model
  for~$\Pcal$} is a triple~$(f,Y,X)$ consisting of a morphism
$f:\PP^1\to\PP^1$ and subsets $Y\subset X\subset\PP^1$ such that the
following diagram commutes:
\begin{equation}
\label{eqn:WP1VP1}
\begin{CD}
  \Wcal @>i>> \PP^1 \\
  @V \Phi VV @V f VV \\
  \Vcal @>i>> \PP^1 \\
\end{CD}
\end{equation}
We say that~$(f,Y,X)$ is a $\bullet$-\emph{model} if in addition
\[
  e_f\bigl(i(w)\bigr) \ge \e(w)\quad\text{for all $w\in\Wcal$;}
\]
and similarly we say that~$(f,Y,X)$ is a $\circ$-\emph{model} if
\[
  e_f\bigl(i(w)\bigr) = \e(w)\quad\text{for all $w\in\Wcal$,}
\]
\end{definition}

With this formalism, we can now define
our three Shafarevich-type sets.

\begin{definition}
Let $\Pcal=(\Wcal,\Vcal,\Phi,\e)$ be a portrait and let $n=\wt(\Pcal)$. Then
\begin{align*}
  \GR_d^1[\Pcal]^\bullet(K,S)
  &= \bigl\{ (f,Y,X) \in \GR_d^1[n](K,S) : \text{$(f,Y,X)$ is a $\bullet$-model for $\Pcal$} \bigr\}, \\
  \GR_d^1[\Pcal]^\circ(K,S)
  &= \bigl\{ (f,Y,X) \in \GR_d^1[n](K,S) : \text{$(f,Y,X)$ is a $\circ$-model for $\Pcal$} \bigr\}, \\
  \GR_d^1[\Pcal]^\star(K,S)
  &= \left\{ (f,Y,X) \in \GR^1_d[n](K,S) :
  \begin{tabular}{@{}l@{}}
  $e_{\tilde f_\gp}\bigl(\widetilde{i(w)\bmod\gp}\bigr) = \e(w)$\\
  for all $w\in\Wcal$ and all $\gp\notin S$\\
  \end{tabular}
 \right\}.
\end{align*}
\end{definition}

It may happen that a portrait has no models using maps of a given
degree.  For example, if the portrait~$\Pcal$ contains~$4$ fixed
points, then it cannot be modeled by a map of degree~$2$, and
similarly if~$\Pcal$ contains a pair of~$2$-cycles.  In order to
describe more generally the constraints on a model, we set an ad hoc
piece of notation.  (A better definition of
$\Moduli_d^1[\Pcal]^\bullet$ as a $\ZZ$-scheme is given
in~\cite{moduliportrait2017}.)

\begin{definition}
Let $\Pcal$ be a portrait and let $d\ge2$. We define
\[
\Moduli_d^1[\Pcal]^\bullet
:= \left\{ f \in \Moduli_d^1(\Kbar) :
\begin{tabular}{@{}l@{}}
  there exist sets $Y\subseteq X\subseteq\PP^1(\Kbar)$ such\\
  that $(f,Y,X)$ is a $\bullet$-model for $\Pcal$\\
\end{tabular}
\right\}.
\]
\end{definition}

\begin{proposition}
\label{proposition:whenisMPempty}
Let~$d\ge2$, and let~$\Pcal=(\Wcal, \Vcal,\Phi,\e)$ be a portrait such
that~~$\Moduli_d^1[\Pcal]\ne\emptyset$. Then~$\Pcal$ satisfies the
following conditions:
\[
\textup{(I)}\quad\sup_{v\in\Vcal} \sum_{w\in \Phi^{-1}(v)} \e(w) \le d.
\qquad
\textup{(II)}\quad\sum_{w\in\Wcal} \bigl(\e(w)-1\bigr) \le 2d-2.
\]
For all $n\ge1$, 
\[
  \textup{(III$_n$)}\quad
  \#\left\{w\in\Wcal :
  \begin{array}{@{}l@{}}
    \Phi^n(w)=w~\text{and}~\Phi^m(w)\ne w\\\text{for all $m<n$}\\
  \end{array}\right\}
  \le \sum_{m\mid n} \mu\left(\frac{n}{m}\right)(d^m+1).
\]
\textup(Here~$\mu$ is the M\"obius function.\textup)
\end{proposition}
\begin{proof}
Constraint~I comes from the fact that~$f$ is a map of degree~$d$,
Constraint~II follows from the Riemann-Hurwitz formula
$\sum\bigl(e_f(P)-1\bigr)=2d-2$ \cite[Theorem~1.1]{MR2316407}, and
Constraint~III$_n$ from the fact that a degree~$d$ map on~$\PP^1$ has
at most the indicated number of points of exact period~$n$
\cite[Remark~43.]{MR2316407}.
\end{proof}

If we fix a dimension~$N$ and a preperiodic portrait~$\Pcal$ and if we
allow the degree~$d$ to grow, then we expect
that~$\Moduli_d^1[\Pcal]^\bullet$ has exactly the expected dimension,
as in the following conjecture. This is in marked contrast to our
uncertainty regarding the size of~$\shafdim_d^N[\Pcal]^\bullet$
as~$d\to\infty$; cf.\ Question~\ref{question:shafdisc}.

\begin{conjecture}
Let $N\ge1$, and let $\Pcal=(\Wcal,\Vcal,\Phi,\e)$ be a preperiodic portrait.
There is a constant $d_0(\Pcal)$ such that for all $d\ge d_0(\Pcal)$ we have
\begin{align*}
  \dim \overline{\Moduli_d^1[\Pcal]^\bullet}
  &=   \dim\Moduli_d^1 - \sum_{w\in\Wcal} \bigl(\e(w)-1\bigr) \\
  &= 2d - 2 - \wt(\Pcal) + \#\Wcal.
\end{align*}
\end{conjecture}

\begin{remark}
The local conditions used to define $\GR_d^1[\Pcal]^\star(K,S)$ reflect
the viewpoint adopted by Petsche and Stout in~\cite{MR3293730}.  We
note that since~$f$ and~$i(\Vcal)$ are assumed to have good reduction
outside~$S$, there is a well-defined map $\tilde f_\gp:\PP^1\to\PP^1$
defined over the residue field of~$\gp$, and so it makes sense to look
at the ramification indices of~$\tilde f_\gp$ at the $\gp$-reductions
of the points in~$i(\Wcal)$.
\end{remark}

\begin{remark}
Since the primary goal of this paper is the study of Shafarevich-type
finiteness theorems, we have been content to define our sets of good
reduction purely as sets. In a subsequent
paper~\cite{moduliportrait2017} we will take up the more refined
question of constructing moduli spaces for dynamaical systems with
portraits, after which the results of the present paper can be
reinterpreted as characterizing the $S$-integral points on these
spaces, with the caveat that there may be field-of-moduli versus
field-of-definition issues.
\end{remark}

Since our goal is to understand the size of the various sets of good
reduction triples~$(f,Y,X)$, we are prompted to make the following
definitions.

\begin{definition}
Let $x\in\{\bullet,\circ,\star\}$. The associated \emph{Shafarevich dimension} is the quantity
\[
  \shafdim_d^1[\Pcal]^x
  =\sup_{\substack{\text{$K$ a number field}\\\text{$S$ a finite set of places}\\}}
  \dim \overline{\operatorname{Image} \Bigl( \GR_d^1[\Pcal]^x(K,S) \to \Moduli_d^1 \Bigr)}.
\]
\end{definition}

We record some elementary properties for future reference.

\begin{proposition}
\label{proposition:circsubsch}
Let $d\ge2$, and let $\Pcal=(\Wcal,\Vcal,\Phi,\e)$ be a portrait.
\begin{parts}
\Part{(a)}
Let $\e':\Vcal\to\NN$ be a weight function satisfying $\e'\ge\e$, let
$\Pcal'=(\Wcal,\Vcal,\Phi,\e')$, and let
$x\in\{\bullet,\circ,\star\}$.  Then
\[
  \GR_d^1[\Pcal']^x(K,S) \subseteq \GR_d^1[\Pcal]^x(K,S).
\]
\Part{(b)}
We have
\[
  \GR_d^1[\Pcal]^\star(K,S) \subseteq \GR_d^1[\Pcal]^\circ(K,S)
  \subseteq \GR_d^1[\Pcal]^\bullet(K,S).
\]
\Part{(c)}
We have
\[
  \shafdim_d^1[\Pcal]^\star \le \shafdim_d^1[\Pcal]^\circ \le
  \shafdim_d^1[\Pcal]^\bullet \le \dim\Moduli_d^1 = 2d-2.
\]
\end{parts}
\end{proposition}
\begin{proof}
(a) and (b) are clear from the definitions of the various sets of good reduction,
and~(c) follows~(b) and the definition of Shafarevich dimension.
We note that if a map~$f$ has good reduction at~$\gp$, then its ramification index
can only increase when~$f$ is reduced modulo~$\gp$.
\end{proof}

\begin{example}
Consider the following two preperiodic portraits:
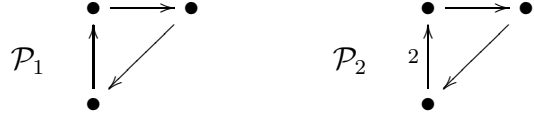
\[
  \raise-20pt\hbox{$\Pcal_{1}$} \quad
  \xymatrix{ {\bullet} \ar[r]_{} & {\bullet} \ar[dl]_{}\\
     {\bullet}  \ar[u]_{}\\
  }
  \qquad\qquad
  \raise-20pt\hbox{$\Pcal_{2}$} \quad
  \xymatrix{ {\bullet} \ar[r]_{} & {\bullet} \ar[dl]_{}\\
     {\bullet}  \ar[u]^{2}_{}\\
  }
\]
We note that the portrait~$\Pcal_2$ is strictly larger than the portrait~$\Pcal_1$
in the sense of Proposition~\ref{proposition:circsubsch}(a),
so the proposition tells us that
$\GR_d^1[\Pcal_2]^\circ(K,S)\subseteq\GR_d^1[\Pcal_1]^\circ(K,S)$. However, we will see in
Section~\ref{section:wt4deg2P1} that if~$\#R_S^*=\infty$, then
\[
  \#\GR_2^1[\Pcal_2]^\circ(K,S)=\infty
  \quad\text{and}\quad
  \#\GR_2^1[\Pcal_1]^\circ(K,S)<\infty.
\]
In words, there are only finitely many degree~$2$ rational maps with
good reduction outside~$S$ that have an \emph{unramified} good
reduction $3$-cycle, but if we allow one of the points in the
$3$-cycle to be ramified, then there are infinitely many such maps.
In terms of Shafarevich dimensions, we have
$\shafdim_d^1[\Pcal_1]^\circ=0$ and $\shafdim_d^1[\Pcal_2]^\circ=1$.
On the other hand, we will show that with the more restrictive
Petsche--Stout good reduction criterion, we have
$\shafdim_d^1[\Pcal_1]^\star=\shafdim_d^1[\Pcal_2]^\star=0$.
Another example of this phenomenon, where more ramification leads to
more maps of good reduction, is given by portraits~$\Pcal_{3,3}$
and~$\Pcal_{4,7}$ in Tables~\ref{table:wt3deg2P1}
and~\ref{table:wt4deg2P1}, respectively.
\end{example}

\section{Good Reduction for Preperiodic Portraits of Weight ${}\le4$ for Degree 2 Maps of $\PP^1$}
\label{section:wt4deg2P1}

We know from Theorem~\ref{theorem:shafconjP1} with $N=1$ and $d=2$
that if a portrait~$\Pcal$ satisfies $\wt(\Pcal)\ge5$, then
$\shafdim_2^1[\Pcal]^\bullet=0$. In other words, dynamical Shafarevich finiteness
holds for degree~$2$ maps $f:\PP^1\to\PP^1$ that model a
portrait~$\Pcal$ of weight at least~$5$.  In this section we give a
complete analysis of preperiodic portraits of weights~$1$ to~$4$.  For example,
it turns out that there are~$22$ such portraits of weight~$4$, and 
dynamical Shafarevich finiteness holds for some of them, but not for
others. For notational convenience, we label portraits
as~$\Pcal_{w,m}$, where~$w$ is the weight and $m\in\{1,2,3\ldots\}$.

\begin{theorem}
\label{theorem:deg2wt4classification}
We consider moduli spaces of degree $2$ maps $\PP^1\to\PP^1$ with
weighted preperiodic portraits.
\begin{parts}
\Part{(a)}
There is~$1$ portrait~$\Pcal$ of weight~$1$ such that
$\Moduli_2^1$ contains a map that can be used to model~$\Pcal$.
\Part{(b)}
There are~$4$ portraits~$\Pcal$ of weight~$2$ such that
$\Moduli_2^1$ contains a map that can be used to model~$\Pcal$.
\Part{(c)}
There are~$8$ portraits~$\Pcal$ of weight~$3$ such that
$\Moduli_2^1$ contains a map that can be used to model~$\Pcal$.
\Part{(d)}
There are~$22$ portraits~$\Pcal$ of weight~$4$ such that
$\Moduli_2^1$ contains a map that can be used to model~$\Pcal$.
\end{parts}
\noindent
These portraits are as catalogued in
Tables~$\ref{table:wt3deg2P1}$,~$\ref{table:wt4deg2P1}$
and~$\ref{table:wt4deg2P1+}$, which also give the values
of the following quantities\textup:
\begin{align*}
  \MD &:= \dim\overline{\Moduli_2^1[\Pcal]^\bullet}, & \SD^\bullet &:= \shafdim_2^1[\Pcal]^\bullet, \\*
  \SD^\circ &:= \shafdim_2^1[\Pcal]^\circ, & \SD^\star &:= \shafdim_2^1[\Pcal]^\star.
\end{align*}
\end{theorem}


\begin{table}[t]
\[
\begin{array}{|c|c|c|c|c|c|c|} \hline
  \# & \Pcal & \wt(\Pcal) & \MD & \SD^\bullet & \SD^\circ & \SD^\star \\ \hline\hline
  \Pcal_{1,1} &
  \xymatrix{ {\bullet}  \ar@(dr,ur)[]_{} } \quad
  & 1 & 2 & 2 & 2 & 2  \TableLoopSpacing\\ \hline\hline
  \Pcal_{2,1} &
  \xymatrix{ {\bullet}  \ar@(dr,ur)[]_{2} } \quad
  & 2 & 1 & 1 & 1 & 1  \TableLoopSpacing\\ \hline
  \Pcal_{2,2} &
  \xymatrix{ {\bullet} \ar[r]_{}& {\bullet}  \ar@(dr,ur)[]_{}} 
  & 2 & 2 & 2 & 2 & 2  \TableLoopSpacing\\ \hline
  \Pcal_{2,3} &
  \xymatrix{ {\bullet}  \ar@(dr,ur)[]_{} } \quad
  \xymatrix{ {\bullet}  \ar@(dr,ur)[]_{} }
  & 2 & 2 & 2 & 2 & 1  \TableLoopSpacing \\ \hline
  \Pcal_{2,4} &
  \xymatrix{
    {\bullet} \ar@(dl,dr)[r]_{}   & {\bullet}   \ar@(ur,ul)[l]_{}  \\
  }
  & 2 & 2 & 2 & 2 & 1 \\ \hline\hline
  \Pcal_{3,1} &
  \xymatrix{ {\bullet} \ar[r]_{} & {\bullet} \ar[dl]_{}\\
     {\bullet}  \ar[u]_{}\\
    }
  & 3 & 2 & 1 & 0 & 0 \\ \hline 
  \Pcal_{3,2} &
  \xymatrix{
    {\bullet} \ar[r]_{}& {\bullet} \ar@(dl,dr)[r]_{}   & {\bullet}   \ar@(ur,ul)[l]_{} \\
  }
  & 3 & 2 & 2 & 2 & 1 \\ \hline
  \Pcal_{3,3} &
  \xymatrix{ {\bullet}  \ar@(dr,ur)[]_{}} \quad
  \xymatrix{
     {\bullet} \ar@(dl,dr)[r]_{}   & {\bullet}   \ar@(ur,ul)[l]_{} \\
  }
  & 3 & 2 & 1 & 0 & 0 \\ \hline
  \Pcal_{3,4} &
  \xymatrix{ {\bullet} \ar[r]_{}& {\bullet} \ar[r]_{}& {\bullet}  \ar@(dr,ur)[]_{}} 
  & 3 & 2 & 2 & 2 & 1  \TableLoopSpacing \\\hline
  \Pcal_{3,5} &
  \xymatrix{ {\bullet}  \ar@(dr,ur)[]_{} } \quad
  \xymatrix{ {\bullet} \ar[r]_{}& {\bullet}  \ar@(dr,ur)[]_{}} 
  & 3 & 2 & 2 & 2 & 1  \TableLoopSpacing \\\hline
  \Pcal_{3,6} &
  \xymatrix{ {\bullet}  \ar@(dr,ur)[]_{} } \quad
  \xymatrix{ {\bullet}  \ar@(dr,ur)[]_{} } \quad
  \xymatrix{ {\bullet}  \ar@(dr,ur)[]_{} }
  & 3 & 2 & 1 & 0 & 0  \TableLoopSpacing \\ \hline
  \Pcal_{3,7} &
  \xymatrix{
     {\bullet} \ar@(dl,dr)[r]^2_{}   & {\bullet}   \ar@(ur,ul)[l]_{} \\
  }
  & 3 & 1 & 1 & 1 & 1 \\ \hline
  \Pcal_{3,8} &
  \xymatrix{ {\bullet}  \ar@(dr,ur)[]_{} } \quad
  \xymatrix{ {\bullet}  \ar@(dr,ur)[]_{2} }
  & 3 & 1 & 1 & 1 & 1 \TableLoopSpacing \\ \hline
\end{array}
\]
\caption{Weight $1$, $2$, and $3$ preperiodic portraits for degree $2$ maps}
\label{table:wt3deg2P1}
\vspace*{45pt}  
\end{table}

\begin{table}[t]
\[
\begin{array}{|c|c|c|c|c|c|c|} \hline
  \# & \Pcal & \wt(\Pcal) & \MD & \SD^\bullet & \SD^\circ & \SD^\star \\ \hline\hline
  \Pcal_{4,1} &
  \xymatrix{ {\bullet}  \ar@(dr,ur)[]_{2} } \quad
  \xymatrix{ {\bullet}  \ar@(dr,ur)[]_{2} }
  & 4 & 0 & 0 & 0 & 0  \TableLoopSpacing\\ \hline 
  \Pcal_{4,2} &
  \xymatrix{
    {\bullet} \ar@(dl,dr)[r]^2_{}   & {\bullet}   \ar@(ur,ul)[l]^2_{}  \\
  }
  & 4 & 0 & 0 & 0 & 0 \\ \hline
  \Pcal_{4,3} &
  \xymatrix{ {\bullet}  \ar@(dr,ur)[]_{2} } \quad
  \xymatrix{ {\bullet}  \ar@(dr,ur)[]_{} } \quad
  \xymatrix{ {\bullet}  \ar@(dr,ur)[]_{} }
  & 4 & 1 & 1 & 1 & 0  \TableLoopSpacing\\ \hline
  \Pcal_{4,4} &
  \xymatrix{ {\bullet}  \ar@(dr,ur)[]_{2} } \quad
  \xymatrix{ {\bullet} \ar[r]_{}& {\bullet}  \ar@(dr,ur)[]_{}} 
  & 4 & 1 & 1 & 1 & 1  \TableLoopSpacing\\ \hline
  \Pcal_{4,5} &
  \xymatrix{ {\bullet}  \ar@(dr,ur)[]_{2} } \quad
  \xymatrix{
    {\bullet} \ar@(dl,dr)[r]_{}  & {\bullet}   \ar@(ur,ul)[l]_{} \\
  }
  & 4 & 1 & 1 & 1 & 0 \\ \hline
  \Pcal_{4,6} &
  \xymatrix{ {\bullet} \ar^2[r]_{}& {\bullet} \ar[r]_{}& {\bullet}  \ar@(dr,ur)[]_{}} 
  & 4 & 1 & 1 & 1 & 1 \TableLoopSpacing \\ \hline
  \Pcal_{4,7} &
  \xymatrix{ {\bullet}  \ar@(dr,ur)[]_{}} \quad
  \xymatrix{
     {\bullet} \ar@(dl,dr)[r]^2_{}   & {\bullet}   \ar@(ur,ul)[l]_{} \\
  }
  & 4 & 1 & 1 & 1 & 0 \\ \hline
  \Pcal_{4,8} &
  \xymatrix{
    {\bullet} \ar[r]_{}& {\bullet} \ar@(dl,dr)[r]^2_{}   & {\bullet}   \ar@(ur,ul)[l]_{} \\
  }
  & 4 & 1 & 1 & 1 & 1 \\ \hline
  \Pcal_{4,9} &
  \xymatrix{ {\bullet} \ar[r]_{} & {\bullet} \ar[dl]_{}\\
     {\bullet}  \ar[u]^{2}\\
    }
  & 4 & 1 & 1 & 1 & 0 \\ \hline
  \Pcal_{4,10} &
  \xymatrix{ {\bullet} \ar[r]_{} & {\bullet}  \ar@(dr,ur)[]_{} } \hspace{.5em}
  \xymatrix{ {\bullet}  \ar@(dr,ur)[]_{} } \hspace{.5em}
  \xymatrix{ {\bullet}  \ar@(dr,ur)[]_{} }
  & 4 & 2 & 0 & 0 & 0  \TableLoopSpacing\\ \hline
  \Pcal_{4,11} &
  \xymatrix{ {\bullet} \ar[r]_{} & {\bullet}  \ar@(dr,ur)[]_{} } \hspace{.5em}
  \xymatrix{ {\bullet} \ar[r]_{} & {\bullet}  \ar@(dr,ur)[]_{} } 
  & 4 & 2 & 1 & 1 & 1  \TableLoopSpacing\\ \hline
  \Pcal_{4,12} &
  \xymatrix{ {\bullet} \ar[r]_{} & {\bullet} \ar[r]_{} & {\bullet}  \ar@(dr,ur)[]_{} } \hspace{.5em}
  \xymatrix{  {\bullet}  \ar@(dr,ur)[]_{} } 
  & 4 & 2 & 0 & 0 & 0  \TableLoopSpacing\\ \hline
  \Pcal_{4,13} &
  \xymatrix{ {\bullet}  \ar@(dr,ur)[]_{}} \hspace{.5em}
  \xymatrix{ {\bullet}  \ar@(dr,ur)[]_{}} \hspace{.5em}
  \xymatrix{
     {\bullet} \ar@(dl,dr)[r]_{}   & {\bullet}   \ar@(ur,ul)[l]_{} \\
  }
  & 4 & 2 & 0 & 0 & 0 \\ \hline
\end{array}
\]
\caption{Weight $4$ preperiodic portraits for degree $2$ maps (Part 1)}
\label{table:wt4deg2P1} 
\vspace*{28pt}  
\end{table}

\begin{table}[t]
\[
\begin{array}{|c|c|c|c|c|c|c|} \hline
  \# & \Pcal & \wt(\Pcal) & \MD & \SD^\bullet & \SD^\circ & \SD^\star \\ \hline\hline
  \Pcal_{4,14} &
  \xymatrix{ {\bullet} \ar[r] &  {\bullet} \ar[r] &  {\bullet} \ar[r] & {\bullet}  \ar@(dr,ur)[]_{} }
  & 4 & 2 & 0 & 0 & 0  \TableLoopSpacing \\ \hline
  \Pcal_{4,15} &
  \xymatrix{ {\bullet} \ar[r] & {\bullet}  \ar@(dr,ur)[]_{} } \quad
  \xymatrix{
     {\bullet} \ar@(dl,dr)[r]_{}   & {\bullet}   \ar@(ur,ul)[l]_{} \\
  } 
  & 4 & 2 & 0 & 0 & 0 \\ \hline
  \Pcal_{4,16} &
  \xymatrix{ {\bullet}  \ar@(dr,ur)[]_{} } \quad
  \xymatrix{
     {\bullet} \ar[r] &  {\bullet} \ar@(dl,dr)[r]_{}   & {\bullet}   \ar@(ur,ul)[l]_{} \\
  } 
  & 4 & 2 & 0 & 0 & 0 \\ \hline
  \Pcal_{4,17} &
  \raisebox{-20pt}{$\displaystyle  \xymatrix{  {\bullet}  \ar@(dr,ur)[]_{} } $ }\quad
  \xymatrix{ {\bullet} \ar[r]_{} & {\bullet} \ar[dl]_{}\\
     {\bullet}  \ar[u]_{}\\
    }
  & 4 & 2 & 0 & 0 & 0 \\ \hline
  \Pcal_{4,18} &
  \xymatrix{
    {\bullet} \ar[dr]_{} \\
    & {\bullet} \ar[r]_{} & {\bullet} \ar@(dr,ur)[]_{} \\
    {\bullet} \ar[ur]_{} \\
  }
  & 4 & 2 & 1 & 1 & 1 \\ \hline
  \Pcal_{4,19} &
  \xymatrix{
    {\bullet} \ar[r]_{} & {\bullet} \ar[r]_{} &
     {\bullet} \ar@(dl,dr)[r]_{}   & {\bullet}   \ar@(ur,ul)[l]_{} \\
  } 
  & 4 & 2 & 0 & 0 & 0 \\ \hline
  \Pcal_{4,20} &
  \xymatrix{
    {\bullet} \ar[r]_{} & 
     {\bullet} \ar@(dl,dr)[r]_{}   & {\bullet}   \ar@(ur,ul)[l]_{} & {\bullet} \ar[l]_{} \\
  } 
  & 4 & 2 & 1 & 1 & 1 \\ \hline
  \Pcal_{4,21} &
  \xymatrix{ {\bullet} \ar[r]_{} & {\bullet} \ar[r]_{} & {\bullet} \ar[dl]_{}\\
     & {\bullet}  \ar[u]_{}\\
    }
  & 4 & 2 & 0 & 0 & 0 \\ \hline
  \Pcal_{4,22} &
  \xymatrix{ {\bullet}  \ar[r]_{} & {\bullet} \ar[d]_{} \\
    {\bullet} \ar[u]_{} & {\bullet} \ar[l]_{} \\
  }
  & 4 & 2 & 0 & 0 & 0 \\ \hline
\end{array}
\]
\caption{Weight $4$ preperiodic portraits for degree $2$ maps (Part 2)}
\label{table:wt4deg2P1+} 
\vspace*{40pt}  
\end{table}

\begin{proof}
Since we will be dealing entirely with preperiodic portraits in this proof,
we write the triple~$(f,X,X)$ as a pair~$(f,X)$.  
For degree~$2$ maps, we see that~$\Moduli_2^1[\Pcal]^\bullet=\emptyset$ unless
the following four conditions are true;
cf.\ Proposition~\ref{proposition:whenisMPempty}.
\par\noindent\begin{tabular}{rl}
(I)& Each point has at most weight~$2$ worth of incoming arrows; \\
(II)& There are at most~$2$ critical points;\\
(III$_1$)& There are at most~$3$ fixed points;\\
(III$_2$)& There is at most one periodic cycle of length~$2$.\\
\end{tabular}

Sublemma~\ref{lemma:fPGL2fX0} says that in order to prove that
$\GR_2^1[\Pcal]^\circ(K,S)/\PGL_2(R_S)$ is finite for all~$K$ and~$S$,
it suffices to prove finiteness after extending~$K$ and
enlarging~$S$. And the definition of~$\shafdim_d^1[\Pcal]^\bullet$ and its
variants is a supremum over all~$K$ and all~$S$.  So we may assume
throughout our discussion that in every model~$(f,X)$ for~$\Pcal$, the
points in~$X$ are in~$\PP^1(K)$, and further that~$S$ is chosen so
that
\[
  \text{$R_S$ is a PID;\qquad $R_S^*$ is infinite;\qquad $2,3\in R_S^*$.}
\]

Using the assumptions that the points in our portraits are in
$\PP^1(K)$ and that~$R_S$ is a PID, Lemma~\ref{sublemma:3ptsmodp} and
the Chinese remainder theorem tell us that we can find an element
of~$\PGL_2(R_S)$ to move three of the points in~$X$ to the
points~$0$,~$1$, and~$\infty$. (Or just to~$0$ and~$\infty$ if
$\#X=2$.)

As in the proof of Proposition~\ref{proposition:simpledN1}, we will
frequently use the Milnor isomorphism~\cite[Theorem~4.5.6]{MR2316407}
\[
  s = (s_1,s_2):\Moduli^1_2\xrightarrow{\;\sim\;}\AA^2,
\]
which we implemented in PARI~\cite{PARI2}, to help distinguish
the~$\PGL_2(\Kbar)$-conjugacy classes of our maps, and we often use
Magma~\cite{MR1484478} to verify that the images of certain maps are Zariski
dense in~$\Moduli_2^1$.

\Pcase{1}{1}{2}
This case was done by Petsche and Stout~\cite[Remark~3]{MR3293730},
but for completeness, we include a proof.
Let $f(x)=(x^2+ax)/(bx+1)$ with $\Resultant(f)=1-ab$,
so $\bigl(f,\{0\}\bigr)\in\GR_2^1[\Pcal_{1,1}]^\bullet(K,S)$ for all $a,b\in R_S$
satisfying $1-ab\in R_S^*$. Further, $f'(0)=a$, so if we
take~$a\in R_S^*$, then~$0$ is not critical modulo~$v$ for all
$v\notin S$.  This suggests that we change variables via
$b=(1-u)a^{-1}$.  Then $f(x)=(ax^2+a^2x)/((1-u)x+a)$ with
$\Resultant(f)=a^4u$ and $f'(0)=a$, so
$\bigl(f,\{0\}\bigr)\in\GR_2^1[\Pcal_{1,1}]^\star(K,S)$ for all
$a,u\in R_S^*$.  The Milnor image of this map in~$\Moduli_2\cong\AA^2$
is
\begin{multline*}
  s\left(\frac{ax^2+a^2x}{(1-u)x+a}\right)
  =\left(\frac{a^2 (u - 1) + 2 a - (u-1)^2}{a u}, \right.\\
 \left. \frac{-a^4 + 2 a^3 - a^2 (u-1)(u-2) - 2 a (u - 1) - (u-1)^2 }{a^2 u}\right).
\end{multline*}
We used Magma to verify that the two rational functions~$s_1(a,u)$
and~$s_2(a,u)$ are algebraically independent in $K(a,u)$. Hence under our
assumption that~$\#R_S^*=\infty$, we see that
$\{s(a,u):a,u\in R_S^*\}$ is Zariski dense in~$\AA^2$.
This completes the proof that $\shafdim_2^1[\Pcal_{1,1}]^\star=2$, and
the other Shafarevich dimensions are also~$2$ by the standard
inequalities in Proposition~\ref{proposition:circsubsch}(e).

\Pcase{2}{1}{1}
Moving the totally ramified fixed point to~$\infty$, the map~$f$ has
the form $f(x)=ax^2+bx+c$. It has good reduction if and only if
$a\in R_S^*$. Then we can conjugate by a map of the form
$x\mapsto a^{-1}x+e$ to put~$f(x)$ in the form $f(x)=x^2+c$.  Since the
ramification at~$\infty$ can't increase when we reduce modulo primes
not in~$S$, we see that
\[
  \bigl(x^2=c,\{\infty\}\bigr) \in \GR_2^1[\Pcal_{2,1}]^\star(K,S)
  \quad\text{for all~$c\in R_S^*$.}
\]
The closure of the image in~$\Moduli_2^1$ is the line~$s_1=2$ of
polynomials.

\Pcase{2}{2}{2}
We move the two points to~$0,\infty$, and then~$f$ has the form
$f(x)=(ax^2+bx+c)/dx$. This map has $\Resultant(f)=acd^2$, so we can
dehomogenize~$d=1$. Thus $f(x)=ax+b+cx^{-1}$ with $ac\in R_S$.
Conjugating by $x\to ux$ gives $u^{-1}f(ux)=ax+bu^{-1}+cu^{-2}x^{-1}$,
so going to~$K(\sqrt{c})$, which is unramified over~$S$, we may
assume that~$c=1$ and~$f(x)=(ax^2+bx+1)/x$.
We also observe that $f^{-1}\bigl(f(\infty)\bigr)=\{0,\infty\}$ and in
$f^{-1}\bigl(f(0)\bigr)=\{0,\infty\}$, so~$0$ and~$\infty$ are
unramified modulo all primes. (Alternatively, one could compute
derivatives, after moving~$\infty$ to a more amenable point.)
Hence
\[
  \bigl(f,\{0,\infty\}\bigr)\in\GR_2^1[\Pcal_{2,2}]^\star(K,S)\quad\text{for
  all $a\in R_S^*$ and $b\in R_S$.}
\] 
The Milnor image is
\begin{multline*}
  s\left(\frac{ax^2+bx+1}{x}\right)
  =\left(\frac{4 a^2 - a b^2 - 2 a + 1}{a}, \right.\\*
 \left. \frac{4 a^3 - a^2 b^2 - 4 a^2 + 5 a - b^2 - 2}{a}\right).
\end{multline*}
We used Magma to verify that the rational functions~$s_1(a,u)$
and~$s_2(a,u)$ are algebraically independent in $K(a,u)$. Hence under
our assumption that~$\#R_S^*=\infty$, we see that $\{s(a,u):a,u\in
R_S^*\}$ is Zariski dense in~$\AA^2$.  This completes the proof that
$\shafdim_2^1[\Pcal_{2,2}]^\star=2$, and the other Shafarevich
dimensions are also~$2$ by the standard inequalities in
Proposition~\ref{proposition:circsubsch}(e).

\Pcase{2}{3}{2}
Moving the two fixed points to~$0$ and~$\infty$, the map~$f$ has the
form $f(x)=(ax^2+bx)/(cx+d)$.  The resultant is
$\Resultant(f)=ad(ad-bc)$. Good reduction implies in particular
that~$a,d\in R_S^*$, so we can dehomogenize~$d=1$ and replace~$f$ with
$af(a^{-1}x)=(x^2+bx)/(a^{-1}cx+1)$. We can also replace~$a^{-1}c$ with~$c$,
so~$f(x)=(x^2+bx)/(cx+1)$ with $\Resultant(f)=1-bc$. 
Hence
\[
  \bigl(f,\{0,\infty\}\bigr)\in\GR_2^1[\Pcal_{2,3}]^\circ(K,S)\quad\text{for
  all $b,c\in R_S$ satisfying $1-bc\in R_S^*$.}
\]
We note that this set of~$(b,c)$ is Zariski dense in~$\AA^2$, under
our assumption that~$\#R_S^*=\infty$. For example, if~$u\in R_S^*$ has
infinite order, then for every $n\ge1$ we can take $b=1-u$ and
$c=1+u+u^2+\cdots+u^n$, and this set of points is Zariski dense.
The Milnor image is
\[
  s\left(\frac{x^2+bx}{cx+1}\right)
  =\left(\frac{-b^2 c - b c^2 + 2}{1-bc}, 
   \frac{-b^2 c^2 - b^2 - b c + 2 b - c^2 + 2 c}{1-bc} \right).
\]
We used Magma to verify that the rational functions~$s_1(a,u)$
and~$s_2(a,u)$ are algebraically independent in $K(a,u)$. Hence under
our assumption that~$\#R_S^*=\infty$, we see that $\{s(a,u):a,u\in
R_S^*\}$ is Zariski dense in~$\AA^2$.  This completes the proof that
$\shafdim_2^1[\Pcal_{2,3}]^\circ=2$.

However, the set $\GR_2^1[\Pcal_{2,3}]^\star(K,S)$ is more restrictive, since
we need the fixed points to be unramified for all primes not in~$S$.
Thus $\bigl(f,\{0,\infty\}\bigr)$ is in this set if and only if
$f'(0)=b\in R_S^*$
and
$f'(\infty)=c\in R_S^*$. We thus need~$b,c,1-bc$ to be $S$-units.
Then $(bc,1-bc)$ is a solution to the $S$-unit equation $u+v=1$, so
there are only finitely many possible values for~$bc$. On the other hand,
any fixed solution~$(u,v)$ gives a map $f(x)=(x^2+bx)/(b^{-1}ux+1)$
satisfying 
\[
  \bigl(f,\{0,\infty\}\bigr)\in\GR_2^1[\Pcal_{2,3}]^\star(K,S)\quad\text{for
  all $b\in R_S^*$.}
\]
Each~$(u,v)$ value gives points lying on a curve in~$\Moduli_2^1$. And there
is at least one such curve, since our assumption that~$2\in S$ says that
we can take~$(u,v)=(-1,2)$, leading to the Milnor image
\[
  s\left(\frac{x^2+bx}{-b^{-1}x+1}\right)
  =\left(\frac{b^2 + 2 b - 1}{2b}, 
   \frac{-b^4 + 2 b^3 - 2 b - 1}{2b^2} \right).
\]
Hence $\shafdim[\Pcal_{2,3}]^\star=1$, a result that was first proven
by Petsche and Stout~\cite[Section~4]{MR3293730}.

\Pcase{2}{4}{2}
We move the two points to~$0$ and~$\infty$, so $f(x)=(ax+b)/(cx^2+dx)$
with $\Resultant(f)=bc(ad-bc)$.  Good reduction implies in particular
that~$b,c\in R_S^*$, so we can dehomogenize~$b=1$. Conjugating~$f$
gives $u^{-1}f(ux)=(aux+1)/(cu^3x^2+du^2x)$. Going to the field
$K(\sqrt[3]{c})$, which is unramified outside~$S$, we can take
$u=c^{-1/3}$ and adjust~$a$ and~$d$ accordingly to put~$f$ in the form
$f(x)=(ax+1)/(x^2+dx)$.
Then
\[
  \bigl(f,\{0,\infty\}\bigr)\in\GR_2^1[\Pcal_{2,4}]^\circ(K,S)\quad\text{for
  all $a,d\in R_S$ with $1-ad\in R_S^*$.}
\]
The map~$f$ is unramified at~$0$ if and only if $d\ne0$
and~$f$ is unramified at~$\infty$ if and only if $a\ne0$.  
The Milnor image is
\[
  s\left(\frac{ax+1}{x^2+dx}\right)
  =\left(\frac{a^3 + 4 a d + d^3 - 6}{1-ad},
  \frac{-2 a^3 - a^2 d^2 - 7 a d - 2 d^3 + 12}{1-ad},
  \right).
\]
We used Magma to verify that the rational functions~$s_1(a,u)$
and~$s_2(a,u)$ are algebraically independent in $K(a,u)$. Hence under
our assumption that~$\#R_S^*=\infty$, we see that $\{s(a,u):a,u\in
R_S^*\}$ is Zariski dense in~$\AA^2$.  This completes the proof that
$\shafdim_2^1[\Pcal_{2,4}]^\circ=2$.

The multiplier of the 2-cycle is $(f^2)'(0)=ad$,
so the points~$0$ and~$\infty$ are unramified modulo all primes
not in~$S$ if and only if~$a,d\in R_S^*$. So in this case $(ad,1-ad)$
is a solution to the $S$-unit equation~$u+v=1$, and each of the
finitely many such solutions yields a family of maps $f(x)=(ax+1)/(x^2+ua^{-1}x)$
with
\[
  \bigl(f,\{0,\infty\}\bigr)\in\GR_2^1[\Pcal_{2,4}]^\star(K,S)\quad\text{for
  all $a\in R_S^*$.}
\]
The Zariski closure of these points form a non-empty finite collection of curves,
since for example $(u,v)=(-1,2)$ gives
\[
  s\left(\frac{ax+1}{x^2-a^{-1}x}\right)
  =\left(\frac{a^6 - 10*a^3 - 1}{2a^3}
  \frac{-a^6 + 9*a^3 + 1}{a^3}
  \right).
\]
Hence $\shafdim[\Pcal_{2,4}]^\star=1$, a result that was first proven
by Petsche and Stout~\cite[Section~5]{MR3293730}.

\Pcase{3}{1}{0}
We first note that almost all rational maps of degree~$2$ have
a~$3$-cycle~\cite[\S6.8]{beardon:gtm}. Hence the image
of~$\Moduli_2^1[\Pcal_{3,1}]^\bullet$ omits only finitely
many points, and thus $\dim\overline{\Moduli_2^1[\Pcal_{3,1}]^\bullet}=2$.
We next move the 3-cycle to $0,1,\infty$, so~$f$ has the form
$f(x)=(ax^2-(a+c)x+c)/(ax^2+ex)$
with $\Resultant(f)=ac(a+e)(c+e)$. We dehomogenize $a=1$.
Then
\[
  \bigl( f, \{0,1,\infty\} \bigr) \in \GR_2^1[\Pcal_{3,1}]^\bullet(K,S)
  \quad\Longleftrightarrow\quad
  c(1+e)(c+e)\in R_S^*.
\]
This leads to solutions to the 4-term $S$-unit equation
\[
  c + (1+e) - (c+e) - 1 = 0.
\]
The multivariable $S$-unit sum theorem~\cite{MR766298,MR1119694} says that
there are finitely many solutions with no subsum equal to~$0$. Ignoring those
finitely many solutions, there are three subsum~$0$ cases:
\begin{parts}
\Part{(1)} $c+(1+e)=0$, which implies that $e_f(\infty)=2$.
\Part{(2)} $c-(c+e)=0$, which implies that $e_f(0)=2$.
\Part{(3)} $c-1=0$, which implies that $e_f(1)=2$.
\end{parts}
This gives three families of pairs~$(f,X)$
in~$\GR_2^1[\Pcal_{3,1}]^\bullet(K,S)$, but every~$f$ is ramified at one of
the three points in~$X$, so these pairs are not
in~$\GR_2^1[\Pcal_{3,1}]^\circ(K,S)$. Instead, they are
in~$\GR_2^1[\Pcal_{4,9}]^\circ(K,S)$. These three families are in
fact~$\PGL_2(R_S)$-conjugate via permuation of the points
in~$\{0,1,\infty\}$. Taking, say, the~$e=0$ family, we have good
reduction for all $c\in R_S^*$, and the Milnor image is
\[
  s\left(\frac{x^2-(1+c)x+c}{x^2}\right)
  = \left( \frac{-c^3 - 5 c^2 + c - 1}{c^2},
  \frac{2 c^3 + 7 c^2 - 2 c + 1}{c^2}\right).
\]
This proves that $\shafdim_2^1[\Pcal_{3,1}]^\bullet=1$ and
$\shafdim_2^1[\Pcal_{3,1}]^\circ=0$.

\Pcase{3}{2}{2}
We move the three points to~$1,0,\infty$, and then~$f$ has the form
$f(x)=a(x-1)/(bx^2+cx)$. This map has $\Resultant(f)=-a^2b(b+c)$, so we can
dehomogenize $a=1$ and replace~$c$ with~$c-b$. This gives the
map $f(x)=(x-1)/(bx^2+(c-b)x)$ with $\Resultant(f)=bc$. Hence
\[
  \bigl(f,\{1,0,\infty\}\bigr) \in \GR_2^1[\Pcal_{3,2}]^\circ(K,S)
  \quad\Longleftrightarrow\quad  b,c\in R_S^*,
\]
and it is in $\GR_2^1[\Pcal_{3,2}]^\circ(K,S)$ if further~$f$ is not
ramified at the points $\{0,1,\infty\}$. The map~$f$ is never ramified at~$1$,
while its multiplier at the 2-cycle is $(f^2)'(0)=(b-c)/c$.
The Milnor image is
\begin{multline*}
  s\left(\frac{x-1}{bx^2+(c-b)x}\right)
  = \left( \frac{b^3 - 3 b^2 c - 2 b^2 + 3 b c^2 - 4 b c + b - c^3}{bc}, \right.\\*
  \left.\frac{-2 b^3 + 6 b^2 c + 4 b^2 - 6 b c^2 + 9 b c - 2 b + 2 c^3 - c^2}{bc}\right).
\end{multline*}
We used Magma to verify that the rational functions~$s_1(b,c)$
and~$s_2(b,c)$ are algebraically independent in $K(b,c)$.
Hence under our assumption that~$\#R_S^*=\infty$, we find
that $\shafdim_2^1[\Pcal_{3,2}]^\circ=2$.

However, if we also require the reduction of~$f$ to be unramified
at~$\{0,1,\infty\}$ for all primes not in~$S$, then we must also
require that $b-c\in R_S^*$. Then $(c/b,1-c/b)$ is a solution to
the $S$-unit equation $u+v=1$, so there are only finitely many
choices for the ratio~$c/b$. For each such choice, say~$c=ub$
with~$u$ fixed, the image in~$\Moduli_2^1$ lies on a curve. And taking,
say, $u=-1$ gives the set of points
\[
  s\left(\frac{x-1}{bx^2-2bx}\right)
  = \left( \frac{-8b^2-2b-1}{b}
  \frac{16b^2+6b+2}{b}\right),\quad b\in R_S^*.
\]
The Zariski closure of this set in $\Moduli_2^1$ is a curve, more
precisely, it is the line $2s_1+s_2=2$. Hence
$\shafdim_2^1[\Pcal_{3,2}]^\star=1$.

\Pcase{3}{3}{0}
We move the fixed point to~$\infty$ and the $2$-cycle to $\{0,1\}$,
which puts~$f$ into the form $f(x)=(x-1)(ax+b)/(cx-b)$. The resultant
is $\Resultant(f)=ab(a+c)(b-c)$, so we may dehomogenize $b=1$.
This puts~$f$ in the form $f(x)=(x-1)(ax+1)/(cx-1)$ with resultant
$\Resultant(f)=a(a+c)(1-c)$. Thus~$f$ has good reduction if and only
if $a,a+c,1-c\in R_S^*$, which gives a solution to the 4-term $S$-unit
equation
\[
  a - (a+c) - (1-c) + 1 = 0.
\]
The multivariable $S$-unit sum theorem~\cite{MR766298,MR1119694} says that
there are finitely many solutions with no subsum equal to~$0$. Ignoring those
finitely many solutions, there are three subsum~$0$ cases:
\begin{parts}
\Part{(1)} $a-(a+c)=0$, which implies that $e_f(\infty)=2$.
\Part{(2)} $a-(1-c)=0$, which implies that $e_f(0)=2$.
\Part{(3)} $a+1=0$, which implies that $e_f(1)=2$.
\end{parts}
This gives three families of pairs~$(f,X)$
in~$\GR_2^1[\Pcal_{3,3}]^\bullet(K,S)$, but every~$f$ is ramified at one of
the three points in~$X$, so these pairs are not
in~$\GR_2^1[\Pcal_{3,3}]^\circ(K,S)$.
Instead, they are ~$\GR_2^1[\Pcal_{4,5}]^\circ(K,S)$ in case~(1) and
in~$\GR_2^1[\Pcal_{4,7}]^\circ(K,S)$ in cases~(2) and~(3).
These give sets of points whose closures are curves:
\begin{align*}
  \Pcal_{4,5}: && 
  s\bigl(-ax^2+(a-1)x+1\bigr)
  &= \left( 2,-a^2-3\right),\quad a\in R_S^*, \\
  \Pcal_{4,7}: && 
  s\left(\frac{-x^2+2x-1}{cx-1}\right)
  &= \left( \frac{-c^3+2}{(c-1)^2}
  \frac{2c^3-4}{(c-1)^2}\right),\quad c\in R_S^*.
\end{align*}
More precisely, they give the curves $s_1=2$ and $2s_1+s_2=0$.  This
completes the proof that $\shafdim_2^1[\Pcal_{3,3}]^\bullet=1$ and
$\shafdim_2^1[\Pcal_{3,3}]^\circ=0$.

\Pcase{3}{4}{1}
We move the three points to~$1,0,\infty$, and then~$f$ has the form
$f(x)=(x-1)(ax+b)/cx$. This map has $\Resultant(f)=-abc^2$, so we can
dehomogenize $c=1$. Then $f(x)=(x-1)(ax+b)/x$ has good reduction if
and only if~$a,b\in R_S^*$.  The multiplier at the fixed point is
$f'(\infty)=a^{-1}$, so~$f$ is not ramified at~$\infty$, and similarly
since $f^{-1}(f(0))=\{0,\infty\}$, the map~$f$ is not ramified
at~$0$. And these statements are true even modulo primes not in~$S$.
Finally we observe that~$f'(1)=a+b$, so~$f$ is ramified at~$1$ if and
only if $a+b=0$.
The Milnor image is
\begin{multline*}
  s\left( \frac{(x-1)(ax+b)}{x} \right)  
  = \left( \frac{a^3 + 2a^2b + ab^2 - 2ab + b}{ab}, \right. \\*
  \left. \frac{a^4 + 2a^3b + a^2b^2 - 4a^2b + a^2 + 3ab + b^2 - 2b}{ab} \right),
\end{multline*}
We used Magma to verify that the rational functions~$s_1(b,c)$
and~$s_2(b,c)$ are algebraically independent in $K(b,c)$.
Hence under our assumption that~$\#R_S^*=\infty$, we find
that $\shafdim_2^1[\Pcal_{3,4}]^\circ=2$.

However, if we also require the reduction of~$f$ to be unramified
at~$\{0,1,\infty\}$ for all primes not in~$S$, then we must also
require that $a+b\in R_S^*$. Then $(-b/a,1+b/a)$ is a solution to
the $S$-unit equation $u+v=1$, so there are only finitely many
choices for the ratio~$b/a$. For each such choice, say~$b=ua$
with~$u$ fixed, the image in~$\Moduli_2^1$ lies on a curve. And taking,
say, $u=1$ gives the set of points
\[
  s\left(\frac{a(x^2-1)}{x}\right)
  = \left( \frac{4a^2 - 2a + 1}{a},
  \frac{4a^3 - 4a^2 + 5a - 2}{a}\right),\quad a\in R_S^*.
\]
The Zariski closure in $\Moduli_2^1$ is a curve. Hence
$\shafdim_2^1[\Pcal_{3,4}]^\star=1$.

\Pcase{3}{5}{2}
We move the three points to~$0,1,\infty$, which puts~$f$ in the form
$f(x)=(ax^2+(b-a)x)/c(x-1)$ with $\Resultant(f)=abc^2$.  We
dehomogenize $c=1$, so $f(x)=(ax^2+(b-a)x)/(x-1)$.  We have
$f'(\infty)=a^{-1}$ and $f^{-1}(f(1))=\{1,\infty\}$, so~$a\in R_S^*$
implies that~$f$ is unramified at~$\infty$ and at~$1$, even modulo
primes not in~$S$. Further, $f'(0)=a-b$, so~$f$ is unramified at~$0$
if and only if~$a\ne b$.
The Milnor image is
\begin{multline*}
  s\left( \frac{ax^2+(b-a)x)}{x-1} \right)  
  = \left( \frac{-a^3 + 2 a^2 b + 2 a^2 - a b^2 - a + b}{ab}, \right. \\ 
  \left. \frac{-a^4 + 2 a^3 b + 2 a^3 - a^2 b^2 - 2 a^2 b - 2 a^2 + 3 a b + 2 a - b^2 - 1}{ab} \right).
\end{multline*}
We used Magma to verify that the rational functions~$s_1(a,b)$
and~$s_2(a,b)$ are algebraically independent in $K(a,b)$.
Hence under our assumption that~$\#R_S^*=\infty$, we find
that $\shafdim_2^1[\Pcal_{3,5}]^\circ=2$.

However, if we also require the reduction of~$f$ to be unramified
at~$\{0,1,\infty\}$ for all primes not in~$S$, then we must also
require that $a-b\in R_S^*$. Then $(b/a,1-b/a)$ is a solution to
the $S$-unit equation $u+v=1$, so there are only finitely many
choices for the ratio~$b/a$. For each such choice, say~$b=ua$
with~$u$ fixed, the image in~$\Moduli_2^1$ lies on a curve. And taking,
say, $u= -1$ gives the set of points
\[
  s\left(\frac{a(x^2-2x)}{x-1}\right)
  = \left( \frac{4a^2-2a+2}{a},
  \frac{4 a^4 - 4 a^3 + 6 a^2 - 2 a + 1}{a^2}\right),\; a\in R_S^*.
\]
The Zariski closure in $\Moduli_2^1$ is a curve, so
$\shafdim_2^1[\Pcal_{3,5}]^\star=1$.

\Pcase{3}{6}{0}
We move the three fixed points to~$0,1,\infty$, so~$f$ has the form
$f(x)=(ax^2+bx)/((a-c)x+b+c)$ with $\Resultant(f) = ac(a+b)(b+c)$. We
dehomogenize~$a=1$, so $f(x)=(x^2+bx)/((1-c)x+b+c)$, and we compute
the three multipliers: $f'(0)=b/(b+c)$, $f'(1)=(b+c+1)/(b+c)$,
$f'(\infty)=1-c$.
We have
\[
  \bigl(f,\{0,1,\infty\}\bigr)\in\GR_2^1[\Pcal_{3,6}]^\bullet(K,S)
  \quad\Longleftrightarrow\quad c,1+b,b+c\in R_S^*.
\]
These maps give a solution to the 4-term $S$-unit equation
\[
  (b+c)-c-(1+b)+1=0.
\]
The multivariable $S$-unit sum theorem~\cite{MR766298,MR1119694} says that
there are finitely many solutions with no subsum equal to~$0$. Ignoring those
finitely many solutions, there are three subsum~$0$ cases:
\[
\begin{array}{ccccc}
  (b+c)-c = 0 &\Longrightarrow& f(x)=\dfrac{x^2}{(1-c)x+c} &\Longrightarrow& e_f(0)=2, \\
  (b+c)-(1+b)=0 &\Longrightarrow& f(x)=\dfrac{x^2+bx}{b+1} &\Longrightarrow& e_f(\infty)=2,\\
  (b+c)+1=0 &\Longrightarrow& f(x)=\dfrac{x^2+bx}{(b+2)x-1} &\Longrightarrow& e_f(1)=2.\\
\end{array}
\]
This proves that~$\shafdim_2^1[\Pcal_{3,6}]^\circ=0$, since the subsum~$0$ cases
have a ramified point, and hence are actually in~$\GR_2^1[\Pcal_{4,3}]^\circ(K,S)$.
The closure of these maps in~$\Moduli_2^1$ is a finite set of curves, since for
example the family with $c=1$ gives the family of polynomials
$f(x)=(x^2+bx)/(b+1)$ whose closure in~$\Moduli_2$ for $b+1\in R_S^*$ is the line $s_1=2$.
This proves that~$\shafdim_2^1[\Pcal_{3,6}]^\bullet=1$, and also (for future reference)
that~$\shafdim_2^1[\Pcal_{4,3}]^\circ=1$.

\Pcase{3}{7}{1}
Moving the two points to~$0$ and~$\infty$ with~$0$ critical, the
map~$f$ has the form $f(x)=(ax+b)/cx^2$ with $\Resultant(f)=b^2c^2$.
Dehomogenizing~$c=1$ gives the map~$f(x)=(ax+b)/x^2$, which has good
reduction if and only if~$b\in R_S^*$.
We conjugate $u^{-1}f(ux)$ with $u=\sqrt[3]{b}$, which is okay since $K(\sqrt[3]{b})$
is unramified outside~$S$. This puts~$f$ into the 
form $f(x)=(ax+1)/x^2$ with $\Resultant(f)=1$.
We also note that~$f$ is ramified at~$\infty$ if and only if~$a=0$,
so taking~$a\in R_S^*$ gives maps such that~$\infty$ is unramified modulo
all primes not in~$S$.
This map has Milnor coordinates
\[
  s\left( \frac{ax+1}{x^2}\right) = (a^3-6,-2a^3+12),
\]
so taking the Zariski closure for $a\in R_S^*$ gives the line
$2s_1+s_2=0$. Hence
$\shafdim_2^1[\Pcal_{3,7}]^\bullet=\shafdim_2^1[\Pcal_{3,7}]^\star=1$.

\Pcase{3}{8}{1} Moving the totally ramified fixed point to~$\infty$
and the other fixed point to~$0$, we have $f(x)=ax^2+bx$ with
$\Resultant(f)=a^2$.  Conjugating by~$x\mapsto a^{-1}x$ puts~$f$ into
the form $f(x)=x^2+bx$, and then $\bigl(f,\{0,\infty\}\bigr)$ is in
$\GR_2^1[\Pcal_{3,8}]^\circ(K,S)$ for all $b\in R_S$ with $b\ne0$, and
$\GR_2^1[\Pcal_{3,8}]^\star(K,S)$ for all $b\in R_S^*$.  The Zariski
closure of the Milnor image of these maps in~$\Moduli_2^1\cong\AA^2$
is the line $s_1=2$.  Hence
$\shafdim_2^1[\Pcal_{3,8}]^\bullet=\shafdim_2^1[\Pcal_{3,8}]^\star=1$.

This completes our analysis of the~$13$ portraits of weights~$1$,~$2$,
and~$3$ in Table~\ref{table:wt3deg2P1}. We move on to analyzing
the~$22$ portraits of weight~$4$ in Tables~\ref{table:wt4deg2P1}
and~\ref{table:wt4deg2P1+}.

\Pcase{4}{1}{0} 
Moving the two totally ramified fixed points to~$0$ and~$\infty$, the
map has the form $f(x)=ax^2$. Good reduction forces~$a\in R_S^*$, and
then conjugation~$af(a^{-1}x)$ yields~$f(x)=x^2$. Hence
$\GR_2^1[\Pcal_{4,1}]^\bullet(K,S)/\PGL_2(R_S)$ consists of a single element.

\Pcase{4}{2}{0} 
Moving the two totally period 2 points to~$0$ and~$\infty$, the
map has the form $f(x)=ax^{-2}$. Good reduction forces~$a\in R_S^*$, and
then conjugation~$a^{-1}f(ax)$ yields~$f(x)=x^{-2}$. Hence
$\GR_2^1[\Pcal_{4,2}]^\bullet(K,S)/\PGL_2(R_S)$ consists of a single element.

\Pcase{4}{3}{1} 
Moving the fixed ponts to~$0,1,\infty$ with~$\infty$ ramified, the
map~$f$ has the form $f(x)=ax^2+(1-a)x$ with $\Resultant(f)=a^2$.
Conjugating gives $af(a^{-1}x)=x^2+(1-a)x$.  The multipliers at~$0$
and~$1$ are $f'(0)=1-a$ and $f'(1)=3-a$.  The Milnor image is
$s\bigl(x^2+(1-a)x\bigr)=(2,1-a^2)$, so~$a\in R_S^*$ gives a Zariski
dense set of points in the line~$s_1=2$, and the same is true if we
disallow $a=1$ and $a=3$. This proves that
$\shafdim_2^1[\Pcal_{4,3}]^\circ=1$; cf.\ the analysis
of~$\Pcal_{3,6}$.  However, if we also insist that~$0$ and~$1$ are
unramified modulo all primes outside~$S$, then we need $1-a\in R_S^*$
and $3-a\in R_S^*$.  In particular, $(a,1-a)$ is a solution to the
$S$-unit equation $u+v=1$, so there are only finitely many values
of~$a$.  This proves that $\shafdim_2^1[\Pcal_{4,3}]^\star=0$.

\Pcase{4}{4}{1} 
Moving the ramified fixed point to~$\infty$, the unramified fixed point to~$0$,
and the other point to~$1$, we find that~$f$ has the
form $f(x)=ax^2-ax$ with $\Resultant(f)=a^2$. Since $f'(0)=-a$
and $f'(1)=a$, we see that~$f$ is unramified at~$0$ and~$1$ modulo all
primes not in~$S$, and hence
$\bigl(f,\{0,1,\infty\}\bigr)\in\GR_2^1[\Pcal_{4,4}]^\star(K,S)$ for
all $a\in R_S^*$. The Milnor image is $s(ax^2-ax) = (2,-a^2-2a)$, so
$\shafdim_2^1[\Pcal_{4,4}]^\star=1$.

\Pcase{4}{5}{1} 
We move the ramified fixed point to~$\infty$ and the other two points
to~$0$ and~$1$.  Then~$f$ has the form $f(x)=ax^2-(a+1)x+1$
with $\Resultant(f)=a^2$ and Milnor image 
$s(ax^2-(a+1)x+1) = (2,-a^2-3)$. The multiplier for the 2-cycle
is $(f^2)'(0)=1-a^2$. Hence
\[
  \bigl(f,\{0,1,\infty\}\bigr) \in \GR_2^1[\Pcal_{4,5}]^\circ(K,S)
  \quad\Longleftrightarrow\quad
  a\in R_S^*~\text{and}~a\ne\pm1.
\]
In particular, we see that $\shafdim_2^1[\Pcal_{4,5}]^\circ=1$;
cf.\ the analysis of~$\Pcal_{3,3}$.  However, if we also require that
the $2$-cycle be unramified modulo all primes not in~$S$, then we need
$1-a^2\in R_S^*$. This gives solutions $(a,1-a)$ to the $S$-unit
equation $u+v=1$, so there are only finitely many maps, and hence
$\shafdim_2^1[\Pcal_{4,5}]^\star=0$.

\Pcase{4}{6}{1} 
We move the points to~$0,1,\infty$ so that $1\to0\to\infty\to\infty$.
Before imposing the condition that~$f$ is ramified at~$1$, this
put~$f$ in the form $f(x)=(ax^2+bx+c)/ex$ with $a+b+c=0$
and $\Resultant(f)=ace^2$. We dehomogenize $e=1$,
and then setting~$f'(1)=0$, we find that~$f$ has
the form $f(x)=a(x-1)^2/x$.
Then $f'(\infty)=a^{-1}$ and $f^{-1}(f(0))=\{0,\infty\}$,
so~$f$ is unramified at~$0$ and~$\infty$
modulo all primes not in~$S$. This gives
\[
  \bigl(f,\{0,1,\infty\}\bigr) \in \GR_2^1[\Pcal_{4,6}]^\star(K,S)
  \quad\Longleftrightarrow\quad
  a\in R_S^*.
\]
The Milnor image is
\[
  s\left(\frac{a(x-1)^2}{x}\right)
  =\left( \frac{-2a+1}{a}, \frac{-4 a^2 + a - 2}{a} \right).
\]
so the Zariski closure is a curve, and hence
$\shafdim_2^1[\Pcal_{4,6}]^\star=1$.

\Pcase{4}{7}{1} 
We move $0$ to the fixed point and~$\infty$ and~$1$ to the $2$-cycle
with~$\infty$ ramified. Ignoring the ramification at~$\infty$ for the
moment, we find that~$f$ has the form $(ax^2+bx)/(x-1)(ax+c)$.
Then we see that~$f$ is ramified at~$\infty$ if and only if $c=a+b$,
so $f(x)=(ax^2+bx)/(x-1)(ax+a+b)$. We compute $\Resultant(f)=a^2(a+b)^2$,
so we can dehomogenize $a=1$, and for convenience replace~$b$ with~$b-1$, to get 
$f(x)=(x^2+(b-1)x)/(x-1)(x+b)$ with $\Resultant(f)=b^2$.
Further, we see that~$f$ is unramified at~$0$ if and only if $b\ne1$
and~$f$ is unramified at~$1$ if and only if $b\ne-1$. Hence
\[
  \bigl(f,\{0,1,\infty\}\bigr) \in \GR_2^1[\Pcal_{4,7}]^\circ(K,S)
  \quad\Longleftrightarrow\quad
  b\in R_S^*~\text{and}~b\ne\pm1.
\]
The Milnor image of~$f$ is
\[
  s\left( \frac{x^2+(b-1)x}{(x-1)(x+b)} \right)
  =\left(\frac{b^3+3b^2-3b+1}{b}, \frac{-2b^3-6b^2+6b-2}{b} \right).
\]
which proves that $\shafdim_2^1[\Pcal_{4,7}]^\circ=1$. Indeed, we have
again landed on the line $2s_1+s_2=0$; cf.\ the analysis
of~$\Pcal_{3,3}$.  However, if we want~$f$ to be unramfied at~$0$
and~$1$ modulo all primes not in~$S$, then we need~$1\pm b\in
R_S^*$. In particular, $(b,1-b)$ is one of the finitely many solutions
of the $S$-unit equation~$u+v=1$, so
$\shafdim_2^1[\Pcal_{4,7}]^\star=0$.

\Pcase{4}{8}{1} 
We move the $2$-cycle to~$0$ and~$\infty$ with~$0$ ramified
and the other point to~$1$. Then~$f$ has the form $f(x)=a(x-1)/bx^2$
with $\Resultant(f)=a^2b^2$, so we can dehomogenize~$a=1$
to get  $f(x)=(x-1)/bx^2$. Assuming that $b\in R_S^*$, we observe
that~$f$ is unramified at~$1$ and~$\infty$, even modulo primes not in~$S$.
Hence
\[
  \bigl(f,\{0,1,\infty\}\bigr) \in \GR_2^1[\Pcal_{4,8}]^\star(K,S)
  \quad\text{for all $b\in R_S^*$.}
\]
The Milnor image is
\[
  s\left( \frac{x-1}{bx^2}\right) =
  \left( \frac{-6 b + 1}{b},\frac{12 b - 2}{b} \right),
\]
so the Zariski closure in~$\Moduli_2^1$ of $\GR_2^1[\Pcal_{4,8}]^\star(K,S)$
is the line $2s_1+s_2=0$.

\Pcase{4}{9}{1} 
We move the 3-cycle to $1\to 0\to \infty\to 1$ with~$1$ a ramification point.
This puts~$f$ in the form $f(x)=a(x-1)^2/(ax^2+ex)$
with $\Resultant(f)=a^2(a+e)^2$. We dehomogenize $a=1$ and replace~$e$ with~$e-1$ to get 
$f(x)=(x-1)^2/(x^2+(e-1)x)$ with $\Resultant(f)=e^2$.
The fact that~$1$ is a ramification point in a 3-cycle tells us
that $(f^3)'(1)=0$, and one of the other points in the 3-cycle
will also be ramified if and only if $(f^3)''(1)=2(1-e^2)/e=0$.
Hence
\[
  \bigl(f,\{0,1,\infty\}\bigr) \in \GR_2^1[\Pcal_{4,9}]^\circ(K,S)
  \quad\Longleftrightarrow\quad
  e\in R_S^*~\text{and}~e\ne\pm1.
\]
The Milnor image is
\[
  s\left( \frac{(x-1)^2}{x^2+(e-1)x}\right) =
  \left( \frac{e^3 - 5 e^2 - e - 1}{e^2},\frac{-2 e^3 + 7 e^2 + 2 e + 1}{e^2} \right),
\]
so the closure of $\GR_2^1[\Pcal_{4,9}]^\circ(K,S)$ is a curve and
$\shafdim_2^1[\Pcal_{4,9}]^\circ=1$. However, if we want the $3$-cycle
to contain only one ramification point modulo primes not in~$S$, then
we need $e^2-1\in R_S^*$.  This yields solutions $(e,1-e)$ to the
$S$-unit equation $u+v=1$, so there are only finitely many such maps
and $\shafdim_2^1[\Pcal_{4,9}]^\star=0$.

\Pcase{4}{10}{0} 
We move the three fixed points to $0$,~$1$, and~$\infty$, and let the
fourth point be~$\a$ with $f(\a)=0$. Then $f$ has the form
$f(x) = (ax^2 + bx)/(ex+a+b-e)$ with $\a=-b/a$ and
\[
  \Resultant(f) = a(a+b)(a-e)(a+b-e).
\]
We dehomogenize $a=1$, so $f(x) = (x^2 + bx)/(ex+1+b-e)$ and $\a=-b$.
Then
\begin{align*}
  &\text{$\{0,1,\infty,-b\}$ has good reduction}
  \quad\Longleftrightarrow\quad
  b,1+b\in R_S^*, \\*
  &\text{$f$ has good reduction}
  \quad\Longleftrightarrow\quad
  1+b, 1-e, 1+b-e \in R_S^*, \\*
  &\bigl(f(x),\{0,1,\infty,-b\}\bigr)\in\GR_2^1[\Pcal_{4,10}]^\bullet(K,S)
  \quad\Longleftrightarrow\quad \\*
  &\omit\hfill $b, 1+b, 1-e, 1+b-e \in R_S^*.$
\end{align*}
But this means that $(-b,1+b)$ is a solution to the $S$-unit equation
$u+v=1$, so there are only finitely many values for~$b$; and then the
fact that $\bigl(b^{-1}(e-1),b^{-1}(1+b-e)\bigr)$ is also a solution
to the $S$-unit equation proves that there are only finitely many
values for~$e$. This completes the proof that
$\GR_2^1[\Pcal_{4,10}]^\bullet(K,S)/\PGL_2(R_S)$ is finite.

\Pcase{4}{11}{1} 
We move the points so that~$0$ and~$\infty$ are fixed by~$f$ and
$f(1)=0$. This puts~$f$ in the form $f(x)=ax(x-1)/(bx-c)$, with
$\Resultant(f)=a^2c(c-b)$. We dehomogenize~$c=1$, so
$f(x)=ax(x-1)/(bx-1)$. Then $f^{-1}(\infty)=\{\infty,b^{-1}\}$, and
our assumption that we have a good reduction model for~$\Pcal_{4,11}$
requires that~$b^{-1}$ be distinct from~$\{0,1,\infty\}$ for all
primes not in~$S$. Thus~$b^{-1}\in R_S^*$ and $b^{-1}-1\in R_S^*$.
The~$S$-unit equation $u-v=1$ has only finitely many solutions, so
there are finitely many values for~$b$.
We  observe that for thees~$b$ values, the map~$f$ is unramified modulo
all primes not in~$S$, since $f^{-1}(f(0))=f^{-1}(f(1))=\{0,1\}$ and
$f^{-1}(f(\infty))=f^{-1}(f(b^{-1}))=\{\infty,b^{-1}\}$.  We also note
that we can take~$b=2$, since~$2\in R_S^*$ by assumption. Thus for
every~$a\in R_S^*$, we see that $\bigl(ax(x-1)/(2x-1),
\{0,1,2^{-1},\infty\}\bigr)$ is
in~$\GR_2^1[\Pcal_{4,11}]^\star(K,S)$. The The Milnor image is
\[
  s\left( \frac{ax(x-1)}{2x-1}\right) =
  \left( \frac{2 a^2 - 2 a + 4}{a},\frac{a^4 - 2 a^3 + 6 a^2 - 4 a + 4}{a^2} \right),
\]
and hence the Zariski closure of~$\GR_2^1[\Pcal_{4,11}]^\circ(K,S)$
in~$\Moduli_2^1$ is a non-empty finite union of curves.  (We remark
that the pairs~$(f,X)$ studied in
Section~\ref{section:pfdynshafP1wt2d}, when restricted to the
case~$d=2$, have portrait~$\Pcal_{4,11}$.)

\Pcase{4}{12}{0} 
We move the points so that~$0$ and~$\infty$ are fixed by~$f$ and
$f(1)=0$. This puts~$f$ in the form $f(x)=ax(x-1)/(bx-c)$, with
$\Resultant(f)=a^2c(c-b)$. We dehomogenize~$a=1$, so
$f(x)=x(x-1)/(bx-c)$. The portrait~$\Pcal_{4,12}$ includes
a point in $f^{-1}(1)=\{x^2-(1+b)x+c=0\}$, and this point is
in~$K$, since the portrait is assumed to be $\Gal(\Kbar/K)$-invariant.
Thus $(1+b)^2-4c=t^2$ for some $t\in K$. Then
$(1+b+t)(1+b-t)=4c\in R_S^*$, so if we have a good reduction portrait
for~$f$, then $c,c-b,1+b\pm t\in R_S^*$. 
This gives us a $5$-term $S$-unit sum
\[
  (1+b+t) + (1+b-t) + 2(c-b) - 2c - 2 = 0.
\]
There are only finitely many solutions with no subsum equal
to~$0$~\cite{MR766298,MR1119694}, so it remains to analyze the~$10$
cases where some subsum vanishes.

\par\framebox{$(1+b+t)+(1+b-t)=0$}\enspace
So $b=-1$ and $f(x)=-x(x-1)/(x+c)$. Then $c$ and $c+1$ are in~$R_S^*$,
so there are only finitely many choices for~$c$.

\par\framebox{$(1+b\pm t)+2(c-b)=0$}\enspace
So $a-b+2c \pm t=0$. Substituting into $(1+b)^2-4c=t^2$
to eliminate~$t$ yields $b=c(c+2)/(c+1)$, and from that we
find that $c/(b-c)=1+c$. We know that $c,b-c\in R_S^*$, so this
shows that~$1+c\in R_S^*$. But then $(1+c,-c)$ is a solution
to the $S$-unit equation $u+v=1$, so there are only finitely many
possibilities for~$c$.

\par\framebox{$(1+b\pm t)-2c=0$}\enspace
So $1+b\pm t=2c$. Substituting into $(1+b)^2-4c=t^2$
to eliminate~$t$ yields $c^2-bc=0$, so either~$c=0$ or~$c-b=0$.
This contradicts the fact that~$c$ and~$c-b$ are~$S$-units.

\par\framebox{$(1+b\pm t)-2=0$}\enspace
So $1+b\pm t=2$. Substituting into $(1+b)^2-4c=t^2$
to eliminate~$t$ yields $c-b=0$, contratdicting the fact
that~$c-b\in R_S^*$.

\par\framebox{$2(c-b)-2c=0$}\enspace
So $b=0$ and $f(x)=x(x-1)/c$. We have $c\in R_S^*$ and $1-4c=t^2$.  We
write $a=\g u^3$ with $u\in R_S^*$ and~$\g$ chosen from a finite set
of representatives for $R_S^*/(R_S^*)^3$.  Then $(u,t)$ is an
$R_S$-integral point on the genus~$1$ curve $y^2=1-4\g x^3$. Siegel's
theorem\cite[D.9.1]{hindrysilverman:diophantinegeometry} says that
there are only finitely many such points.

\par\framebox{$2(c-b)-2=0$}\enspace
So $c=b+1$ and $f(x)=x(x-1)/(bx-b-1)$. We have $c\in R_S^*$ and
$c^2-4c=t^2$. We write $c=\g u^3$ with $u\in R_S^*$ and~$\g$ chosen
from a finite set of representatives for $R_S^*/(R_S^*)^2$.  Then
$(u,t/u)$ is an $R_S$-integral point on the genus~$1$ curve $y^2=\g^2
x^4-4\g x$.  Siegel's
theorem\cite[D.9.1]{hindrysilverman:diophantinegeometry} says that
there are only finitely many such points.

\par\framebox{$-2c-2=0$}\enspace
So $c=-1$ and $f(x)=x(x-1)/(bx+1)$. We have $1+b\in R_S^*$ and
$(1+b)^2+4=t^2$.  We write $1+b=\g u^2$ with $u\in R_S^*$ and~$\g$
chosen from a finite set of representatives for $R_S^*/(R_S^*)^2$.
Then $(u,t)$ is an $R_S$-integral point on the genus~$1$ curve
$y^2=\g^2u^4+4$. Siegel's
theorem\cite[D.9.1]{hindrysilverman:diophantinegeometry} says that
there are only finitely many such points.

\Pcase{4}{13}{0} 
We move the $2$-cycle to~$0,\infty$, so $f(x) = (ax+b)/(cx^2+dx)$. The
resultant is $-bc(ad-bc)$, so we can dehomogenize $c=1$. Moving a
fixed point to~$1$, we have $a+b=d+1$, so $f(x) =
(ax+b)/(x^2+(a+b-1)x)$ with $\Resultant(f)=-b(a-1)(a+b)$. The good
reduction assumption for~$f$ tells us that $b,a-1,a+b\in R_S^*$, so we
obtain a $4$-term $S$-unit equation
\[
  (a+b) - (a-1) - b - 1 = 0.
\]
The multivariable $S$-unit sum theorem~\cite{MR766298,MR1119694} says that
there are finitely many solutions with no subsum equal to~$0$. Ignoring those
finitely many solutions, there are three subsum~$0$ cases:
\begin{parts}
\Part{(1)} $(a+b)-(a-1)=0$, so $b=-1$.  
\Part{(2)} $(a+b)-b=0$, so $a=0$.  
\Part{(3)} $(a+b)-1=0$, so $a=1-b$.  
\end{parts}
The portrait~$\Pcal_{4,13}$ has a second fixed point. The fixed points
of~$f$ are the roots of
\[
  (x-1)(x^2+(a+b)x+b) = 0.
\]
We have assumed that the points in~$\Pcal_{4,13}$ are defined over~$K$, so the
quadratic has a root in~$K$. Thus there is a $t\in K$ such that
\[
  (a+b)^2-4b=t^2.
\]
And since~$a,b\in R_S$, we have~$t\in R_S$. From earlier we know
that~$a+b$ and~$b$ are in~$R_S^*$, so we can write $a+b=\g u^2$ and
$b=\d v^4$, with $u,v\in R_S^*$ and~$\g,\d$ chosen from a finite set
of representatives for $R_S^*/(R_S^*)^4$. Then $(uv^{-1},tv^{-2})$ is
a $R_S$-integral point on the genus~$1$ curve $y^2=\g^2x^4-4\d$.
Siegel's theorem\cite[D.9.1]{hindrysilverman:diophantinegeometry} says
that there are only finitely many such points. Hence there are only
finitely many possibilities for the ratio~$u/v$, and thus only
finitely many possibilities for $\g^2\d^{-1}(u/v)^4 = (a+b)^2/b$.  But
we know from the three cases described earlier that either $b=-1$ or
$a=0$ or $a=1-b$. Substituting these into $(a+b)^2/b$, we find that
there are finitely many values for, respectively,~$-(a-1)^2$,~$b$,
and~$1/b$. 

\Pcase{4}{14}{0} 
We move the points to $\a\to1\to0\to\infty$ with~$\infty$ fixed. Ignoring~$\a$
for the moment, this means that~$f$ has the form
$f(x)=(ax^2-(a+c)x+c)/ex$. We have $\Resultant(f)=-ace^2$, so good
reduction forces $a,c,e\in R_S^*$.  We dehomogenize by setting
$e=1$. At this stage the pair $\bigl(f,\{0,1,\infty\}\bigr)$ has good
reduction. However, we need to adjoin the point~$\a$ to the set~$X$.
The point~$\a$ is a root of the numerator of~$f(x)-1$, so~$\a$ is
a root of the polynomial
\begin{equation}
  \label{eqn:ax2ac1xc}
  a x^2 - (a+c+1) x + c = 0.
\end{equation}
Since we are assuming that $\a\in K$, the discriminant of this
quadratic polynomial is a square in~$K$, say
\[
  t^2 = (a+c+1)^2-4ac \quad\text{with}\quad t\in R_S.
\]
Then
\[
  (a+c+1+t)(a+c+1-t) = 4ac \in R_S^*,
\]
so $a+c+1\pm t\in R_S^*$.  So we now know four $S$-units,
\[
  a,c,a+c+1+t,a+c+1-t\in R_S^*,
\]
which yields a $5$-term $S$-unit sum
\[
   (a+c+1+t) + (a+c+1-t) -2a - 2c - 2 = 0.
\]
There are only finitely many solutions with no subsum equal
to~$0$~\cite{MR766298,MR1119694}, so it remains to analyze the~$10$
cases where some subsum vanishes.

\par\framebox{$(a+c+1+t) + (a+c+1-t)=0$}\enspace
Substituting $c=-a-1$, we find that $t^2 = -4a(a-1)$.
Since~$a\in R_S^*$, we may write $a=\g u^3$ with $u\in R_S^*$ and~$\g$
chosen from a finite set of representatives for $R_S^*/(R_S^*)^3$.
Then $(u,tu^{-1})$ is an $S$-integral point on the genus~$1$ curve
$y^2=-4\g^2x^4+4\g^2 x$. Siegel's
theorem\cite[D.9.1]{hindrysilverman:diophantinegeometry} says that
there are only finitely many such points.

\par\framebox{$(a+c+1\pm t) - 2a =0$}\enspace
Then $0=(a+c+1)^2-4ac-t^2=4a$, contradicting $a\in R_S^*$.

\par\framebox{$(a+c+1\pm t) - 2c =0$}\enspace
Then $0=(a+c+1)^2-4ac-t^2=4c$, contradicting $c\in R_S^*$.

\par\framebox{$(a+c+1\pm t) - 2 =0$}\enspace
Then
\[
  0=(a+c+1)^2-4ac-t^2=4(a+c-ac).
\]
Hence $1=(a-11)(c-1)$, so $a-1$
and $c-1$ are $S$-units.  Thus $(1-a,a)$ and $(1-c,c)$ are each
solutions to the $S$-unit equation $u+v=1$, which has finitely many
solutions.

\par\framebox{$-2a-2c=0$}\enspace
Substituting $a=-c$, we find that $t^2=1+4a^2$.  Since~$a\in R_S^*$,
we may write $a=\g u^2$ with $u\in R_S^*$ and~$\g$ chosen from a
finite set of representatives for $R_S^*/(R_S^*)^2$.  Then $(u,t)$ is
an $R_S$-integral point on the genus~$1$ curve $y^2=1+4\g^2
x^4$. Siegel's
theorem\cite[D.9.1]{hindrysilverman:diophantinegeometry} says that
there are only finitely many such points.

\par\framebox{$-2a-2=0$}\enspace
Substituting $a=-1$, we find that
$t^2=c^2+4c$. Since~$c\in R_S^*$,
we may write $c=\g u^3$ with $u\in R_S^*$ and~$\g$ chosen from a
finite set of representatives for $R_S^*/(R_S^*)^3$.
Then $(u,tu^{-1})$ is an $S$-integral point on the genus~$1$
curve $y^2=\g^2x^4+4\g x$. Siegel's
theorem\cite[D.9.1]{hindrysilverman:diophantinegeometry} says that
there are only finitely many such points.

\par\framebox{$-2c-2=0$}\enspace
Substituting $c=-1$, we find that $t^2=a^2+4a$. The analysis is then
identical to the previous case with $-2a-2=0$.

\Pcase{4}{15}{1} 
The portrait $\Pcal_{4,15}$ contains the portrait $\Pcal_{3,3}$ as a
subportrait, and we already proved that
$\shafdim_2^1[\Pcal_{3,3}]^\circ=0$, so the same is true
for~$\Pcal_{4,15}$. On the other hand, if we allow any of the points
in~$\Pcal_{4,15}$ to have weight greater than~$1$, then the total
weight would be at least~$5$, in which case
Theorem~\ref{theorem:shafconjP1}(a) gives us finiteness. Hence
$\shafdim_2^1[\Pcal_{4,15}]^\bullet=0$.

\Pcase{4}{16}{0} 
The portrait $\Pcal_{4,16}$ contains the portrait $\Pcal_{3,3}$ as a
subportrait, and we already proved that
$\shafdim_2^1[\Pcal_{3,3}]^\circ=0$, so the same is true
for~$\Pcal_{4,16}$. On the other hand, if we allow any of the points
in~$\Pcal_{4,16}$ to have weight greater than~$1$, then the total
weight would be at least~$5$, in which case
Theorem~\ref{theorem:shafconjP1}(a) gives us finiteness. Hence
$\shafdim_2^1[\Pcal_{4,16}]^\bullet=0$.

\Pcase{4}{17}{0} 
The portrait $\Pcal_{4,17}$ contains the portrait $\Pcal_{3,1}$ as a
subportrait, and we already proved that
$\shafdim_2^1[\Pcal_{3,1}]^\circ=0$, so the same is true
for~$\Pcal_{4,17}$. On the other hand, if we allow any of the points
in~$\Pcal_{4,17}$ to have weight greater than~$1$, then the total
weight would be at least~$5$, in which case
Theorem~\ref{theorem:shafconjP1}(a) gives us finiteness. Hence
$\shafdim_2^1[\Pcal_{4,17}]^\bullet=0$.

\Pcase{4}{18}{1} 
Moving the four points to $b,1,0,\infty$, we see that
$f(x)=a(x-1)(x-b)/ex$
with $\Resultant(f)=a^2be^2$, so we can dehomogenize~$a=1$.
Then
\[
  \bigl(f(x),\{b,1,0,\infty\}\bigr)\in\GR_2^1[\Pcal_{4,18}]^\bullet(K,S)
  \quad\Longleftrightarrow\quad
  b,1-b,e\in R_S^*.
\]
(Note that~$b,1-b\in R_S^*$ is the condition for $\{b,0,1,\infty\}$ to
have good reduction outside~$S$.) Then~$(b,1-b)$ is a solution to the
$S$-unit equation $u+v=1$, so there are are finitely many values
for~$b$. Each value of~$b$, for example~$b=2$, yields a curve
in~$\Moduli_2^1$, for example, the Milnor image with~$b=2$ is
\[
  s\left( \frac{(x-1)(x-2)}{ex}\right) =
  \left( \frac{2 e^2 - 4 e + 17}{2 e},\frac{-4 e^3 + 19 e^2 - 8 e + 17}{2 e^2} \right).
\]
Hence $\shafdim_2^1[\Pcal_{4,18}]^\circ=1$. However, since
$f^{-1}(f(1))=f^{-1}(f(b))=\{1,b\}$ and
$f^{-1}(f(0))=f^{-1}(f(\infty))=\{0,\infty\}$, we see that~$f$ modulo
primes not in~$S$ is unramified at the points in~$\{b,1,0,\infty\}$,
so the above maps with~$b=2$ and $e\in R_S^*$ are in
$\GR_2^1[\Pcal_{4,18}]^\star(K,S)$, and hence
$\shafdim_2^1[\Pcal_{4,18}]^\circ=1$.

\Pcase{4}{19}{0} 
Moving $0,\infty$ to the $2$-cycle and~$1$ to the incoming point, we
see that $f(x)=a(x-1)/(bx^2+cx)$. This has
$\Resultant(f)=a^2b(b+c)$. We dehomogenize $b=1$, so
$f(x)=a(x-1)/x(x+c)$ with $a,1+c\in R_S^*$.  The fourth point of the
portrait is in~$f^{-1}(1)$, so it is a root of $x^2+(c-a)x+a$. Since
that point is in~$K$ by assumption, we see that the discriminant
$(c-a)^2-4a$ must be a square in~$K$, say equal to~$t^2$. Then
\[
  (c-a+t)(c-a-t)=4a\in R_S^*,
\]
so $c-a\pm t\in R_S^*$. This gives us a $5$-term $S$-unit sum
\[
  (c-a+t) + (c-a-t) - 2(1+c) + 2a + 2 = 0.
\]
There are only finitely many solutions with no subsum equal
to~$0$~\cite{MR766298,MR1119694}, so it remains to analyze the~$10$
cases where some subsum vanishes.

\par\framebox{$-2(1+c)+2a=0$}\enspace
Then $a=c+1$ and $f(x)=a(x-1)/x(x+a-1)$.  We have $1-4a=t^2$.  We
write $a=\g u^4$ with $u\in R_S^*$ and $\g$ chosen from a finite set
of representatives for $R_S^*/(R_S^*)^4$. Then $1-4\g u^4=t^2$, so
$(u,t)$ is an $R_S$-integral point on the genus~$1$ curve $Y^2=1-4\g
X^4$.  Siegel's
theorem~\cite[D.9.1]{hindrysilverman:diophantinegeometry} tells us
that there are only finitely many solutions.

\par\framebox{$-2(1+c)+2=0$}\enspace
Then $c=0$ and $f(x)=a(x-1)/x^2$. This map has $e_f(0)=2$, so
we do not get the portrait~$\Pcal_{4,19}$ in which every point has
multiplicity~$1$.

\par\framebox{$2a+2=0$}\enspace
Then $a=-1$ and $f(x)=(-x+1)/x(x+c)$.  We have $(c+1)^2+4=t^2$.  We
write $c+1=\g u^2$ with $u\in R_S^*$ and $\g$ chosen from a finite set
of representatives for $R_S^*/(R_S^*)^2$. Then $\g^2 u^4 + 4 = t^2$,
so $(u,t)$ is an $R_S$-integral point on the genus~$1$ curve
$Y^2=\g^2 X^4 +4$.  Siegel's
theorem~\cite[D.9.1]{hindrysilverman:diophantinegeometry} tells us
that there are only finitely many solutions.

\par\framebox{$(c-a+t)+(c-a-t)=0$}\enspace
Then $a=c$ and $f(x)=a(x-1)/x(x+a)$. We have $a,1+a\in R_S^*$,
so $(-a,1+a)$ is a solution to the $S$-unit equation $u+v=1$.
Here there are only finitely many choices for~$a$.

\par\framebox{$(c-a\pm t)-2(1+c)=0$}\enspace
Then $\pm t=a+c+2$, and the equation $(c-a)^2-4a=t^2$ becomes
$ac+2a+c+1=0$. We rewrite this as $a(c+1)+(c+1)+1=0$.
Thus $\bigl(a(c+1),c+1\bigr)$ is a solution to the $S$-unit equation
$u+v+1=0$, so has only finitely many solutions.

\par\framebox{$(c-a\pm t)+2a=0$}\enspace
Then $\pm t=a+c$, and the equation $(c-a)^2-4a=t^2$ becomes
$4a(c+1)=0$. This contradicts the fact that~$a$ and~$c$ are
in~$R_S^*$.

\par\framebox{$(c-a\pm t)+2=0$}\enspace
Then $\pm t=c-a+2$, and the equation $(c-a)^2-4a=t^2$ becomes
$4(c+1)=0$, contradicting $c+1\in R_S^*$.

This completes the proof that $\shafdim_2^1[\Pcal_{4,19}]^\circ=0$.
But if we assign a weight greater than~$1$ to any of the points
in~$\Pcal_{4,19}$, then the resulting portrait will have total weight
at least~$5$, so Theorem~\ref{theorem:shafconjP1}(a) gives us finiteness. Hence
$\shafdim_2^1[\Pcal_{4,19}]^\bullet=0$.   

\Pcase{4}{20}{1} 
Moving $0,\infty$ to the $2$-cycle and~$1$ to an incoming point, we
see that $f(x)=a(x-1)/(bx^2+cx)$. This has
$\Resultant(f)=a^2b(b+c)$. In particular, $a,b\in R_S^*$, so we can
dehomogenize $b=1$ and $f(x)=a(x-1)/x(x+c)$ with $a,1+c\in R_S^*$.
The fourth point of the portrait in~$f^{-1}(\infty)$, so it is the
point~$-c$.  Then $\{0,1,\infty,-c\}$ has good reduction if and only
if $c,1+c\in R_S^*$, so $(-c,1+c)$ is a solution to the $S$-unit
equation $u+v=1$. There are thus only finitely many choices
for~$c$. For example, since $2\in R_S^*$, we may could take
$c=1$. Then $a(x-1)/x(x+1)\in\GR_2^1[\Pcal_{4,20}]^\circ(K,S)$ for all
$a\in R_S^*$.  The Milnor image is
\[
  s\left( \frac{a(x-1)}{x(x+1)}\right) =
  \left( \frac{a^2 - 10 a - 1}{2 a},\frac{-a^2 + 9 a + 1}{a} \right),
\]
which shows that the Zariski closure of
$\GR_2^1[\Pcal_{4,20}]^\circ(K,S)$ in $\Moduli_2^1$ is a non-empty
finite union of curves. Further, since
\begin{align*}
  f^{-1}(f(1))&=f^{-1}(f(\infty))=\{1,\infty\},\\
  f^{-1}(f(0))&=f^{-1}(f(-1))=\{0,-1\},
\end{align*}
we see that~$f$ modulo primes not in~$S$ is unramified at the points
in~$\{-1,1,0,\infty\}$, so the above maps with~$c=1$ and $a\in R_S^*$
are in $\GR_2^1[\Pcal_{4,20}]^\star(K,S)$, and hence
$\shafdim_2^1[\Pcal_{4,20}]^\star=1$.  Finally, we note that
$\shafdim_2^1[\Pcal_{4,20}]^\bullet=\shafdim_2^1[\Pcal_{4,20}]^\circ$, since
if we assign a weight greater than~$1$ to any of the points
in~$\Pcal_{4,20}$, then the resulting portrait will have total weight
at least~$5$, so Theorem~\ref{theorem:shafconjP1}(a) gives us
finiteness. Hence $\shafdim_2^1[\Pcal_{4,20}]^\bullet=1$.

\Pcase{4}{21}{0} 
The portrait $\Pcal_{4,21}$ contains the portrait $\Pcal_{3,1}$ as a
subportrait, and we already proved that
$\shafdim_2^1[\Pcal_{3,1}]^\circ=0$, so the same is true
for~$\Pcal_{4,21}$. On the other hand, if we allow any of the points
in~$\Pcal_{4,21}$ to have weight greater than~$1$, then the total
weight would be at least~$5$, in which case
Theorem~\ref{theorem:shafconjP1}(a) gives us finiteness. Hence
$\shafdim_2^1[\Pcal_{4,21}]^\bullet=0$.

\Pcase{4}{22}{0} 
We move three of the points in the 4-cycle to~$0$,~$1$, and~$\infty$, and
we denote the fourth point by~$c$.
The map~$f$ then has the form
\begin{align*}
  f(x) &= \frac{c(x-1)(x+a)}{x\bigl(x-c+(c-1)(c+a)\bigr)},\\
  \Resultant(f) & = ac^2(1-c)(2-c)(a+c)(1-a-c).
\end{align*}
The set $\{c,0,1,\infty\}$ has good reduction outside~$S$ if and only
if~$c,1-c\in R_S^*$. Hence
$\bigl(f,\{c,0,1,\infty\}\bigr)\in\GR_2^1[\Pcal_{4,22}]^\bullet$ if and only
if
\[
  a,c,1-c,2-c,a+c,1-a-c\in R_S^*.
\]
Then $(c,1-c)$ is a solution to the $S$-unit equation $u+v=1$,
so there are only finitely many values of~$c$.
Then the fact that $(a+c,1-a-c)$ is also a solution to the
$S$-unit equation shows that there are only finitely many values of~$a$.
Hence $\shafdim_2^1[\Pcal_{4,22}]^\bullet=0$.

This completes our analysis of the~$22$ weight~$4$ portraits in
Tables~\ref{table:wt4deg2P1} and~\ref{table:wt4deg2P1+}, and with it,
the proof of Theorem~\ref{theorem:deg2wt4classification}.
\end{proof}

\section{Possible Generalizations}
\label{section:generalizations}

It is natural to attempt to general Theorem~\ref{theorem:shafconjP1}(a) to
self-maps of~$\PP^N$ with $N\ge2$. The naive generalization
fails. Indeed, suppose that we define~$\GR_d^N[n](K,S)$ to be the set
of triples~$(f,Y,X)$ such that $f:\PP^N_K\to\PP^N_K$ is a degree~$d$
morphism defined over~$K$ and~$Y\subseteq X\subset\PP^1(\Kbar)$ are
finite sets satisfying the following conditions:\footnote{We note that
  this definiton is not entirely consistant with our definition
  of~$\GR_d^1[n](K,S)$, since we've replaced the earlier ramification
  condition on~$Y$ with the simpler condition that~$Y$ contain~$n$
  points.}
\begin{parts}
  \Part{\textbullet}
  $X=Y\cup f(Y)$.
  \Part{\textbullet}
  $X$ is $\Gal(\Kbar/K)$-invariant.
  \Part{\textbullet}
  $\#Y=n$.
  \Part{\textbullet}
   $f$ and $X$ have good reduction outside $S$.
\end{parts}
Then it is easy to see that for any fixed~$d$ and~$N$, the
set~$\GR_d^N[n](K,S)$ can be infinite for arbitrarily large~$n$.
We illustrate with $d=N=2$, since the general case is then clear.

Consider the family of maps $f_{a,b}:\PP^2\to\PP^2$ defined by
\begin{equation}
  \label{eqn:fabXYZ}
  f_{a,b}(X,Y,Z) = [aXZ+X^2,bYZ+Y^2,Z^2]\quad\text{with $a,b\in R_S$.}
\end{equation}
Then~$f_{a,b}$ has good reduction at all primes~$\gp\notin S$. And it is
not an isotrivial family, since for example the characteristic polynomial
of~$f_{a,b}$ acting on the tangent space at the fixed point~$[0,0,1]$
is easily computed to be~$(T-a)(T-b)$. For a given~$n$,
we take~$K=\QQ$ and we take~$S$ to be the set of primes dividing
$2\prod_{i=1}^n(2^i-1)$, and we let
\[
  X_n := \bigl\{ [1,2^i,0]\in\PP^N(\QQ) : 0\le i\le n \bigr\}.
\]
Then~$X_n$ has good reduction at all~$p\notin S$, and, since
$f\bigl([1,y,0]\bigr)=[1,y^2,0]$, we see that $f_{a,b}(X_{n-1})\subset X_n$.
Hence
\[
  (f_{a,b},X_{n-1},X_n)\in \GR_2^2[n](\QQ,S)/\PGL_3(\ZZ_S)
\]
gives infinitely many inequivalent triples as~$a$ and~$b$ range over~$\ZZ_S$.

One key step in the proof of Theorem~\ref{theorem:shafconjP1}(a) that
goes wrong when we try to generalize to~$\PP^N$ is
Lemma~\ref{lemma:interpratfnc}, which says that if two
maps~$f,g:\PP^1\to\PP^1$ agree at enough points, then~$f=g$. This is
false in higher dimension, and indeed, the maps~$f_{a,b}$ defined
by~\eqref{eqn:fabXYZ} take identical values at all points on the
line~$Z=0$.

This suggests two ways to rescue the theorem.

First, we might simply say that two maps are ``the same'' if they take
the same values on a non-trivial subvariety of~$\PP^N$. This is a
somewhat drastic solution, but the following partial generalization of
Lemma~\ref{lemma:interpratfnc}, whose proof we leave to the reader,
makes it a reasonable solution.

\begin{lemma}
Let $K$ be a field, and let $f,g:\PP_K^N\to\PP_K^N$ be morphisms of
degrees~$d$ and~$e$, respectively. Suppose that
\[
  \#\bigl\{P\in\PP^N(K) : f(P)=g(P)\bigr\} \ge (d+e)^N+1.
\]
Then there is a curve~$C\subset\PP^N_K$ such that $f(P)=g(P)$ for all $P\in C$.
\end{lemma}

Second, we might insist that the marked points in the set~$X$ are in
sufficiently general position to ensure that~$f|_X=g|_X$
forces~$f=g$. Thus writing~$\End_d^N$ for the space of degree~$d$
self-morphisms of~$\PP^N$, we might say that a set~$Y\subset\PP^N$ is
in \emph{$d$-general position for~$\PP^N$} if the map
\[
  \End_d^N\longrightarrow (\PP^N)^{\#Y},\quad
  f \longmapsto \bigl(f(P)\bigr)_{P\in Y}
\]    
is injective. Then a version of Theorem~\ref{theorem:shafconjP1}(a) might be
true if we restrict to triples~$(f,Y,X)\in\GR_d^N[n](K,S)$ for which~$Y$ is
in~$d$-general position for~$\PP^N$.

We will not further pursue these, or other potential, generalizations
of Theorem~\ref{theorem:shafconjP1}(a) to~$\PP^N$ in this paper.

A second possible generalization of our results would be to extend
them to other types of fields, for example taking~$K=k(C)$ to be the
function field of a curve over an algebraically closed field~$k$.
If~$k$ has characteristic~$0$, then much of the argument in this paper
should carry over, although there may be issues with isotrivial maps;
while if~$k$ has characteristic~$p>0$, then issues of wild
ramification arise, as does the fact that the theorem on $S$-unit
equations is more restrictive in requiring more than the simple ``no
vanishing subsum'' condition.  Again, we have chosen not to pursue
such function field generalizations in the present paper.


\begin{acknowledgement}
The author would like to thank Dan Abramovich, Rob Benedetto, Noah
Giansiracusa, Jeremy Kahn, Sarah Koch, and Clay Petsche for their
helpful advice.
\end{acknowledgement}




\def\cprime{$'$}

\end{document}